\documentclass[final,leqno,onefignum,onetabnum]{siamltex1213}

\usepackage{amsmath,amscd,amssymb}
\usepackage{mathtools,bbm}
\usepackage{enumerate}
\usepackage{color,graphics}
\usepackage{epsfig}
\usepackage{psfrag}
\usepackage[caption=false]{subfig}

\newtheorem{remark}[theorem]{Remark}
\newtheorem{example}[theorem]{Example}
\makeatletter
\renewcommand*\env@matrix[1][*\c@MaxMatrixCols c]{%
	\hskip -\arraycolsep
	\let\@ifnextchar\new@ifnextchar
	\array{#1}}
\makeatother

\DeclareMathOperator{\co}{co}
\DeclareMathOperator{\cl}{cl}
\DeclareMathOperator{\sap}{sp}
\DeclareMathOperator{\dist}{dist}
\DeclareMathOperator{\proj}{proj}
\DeclareMathOperator{\eq}{eq}
\DeclareMathOperator*{\argmin}{arg\,min}
\DeclareMathOperator*{\argmax}{arg\,max}
\DeclareMathOperator{\lev}{lev_{\leq\gamma}}
\DeclareMathOperator{\inter}{int}
\DeclareMathOperator{\bd}{bd}
\DeclareMathOperator{\circu}{Circ}
\DeclareMathOperator{\trid}{Trid}

\newcommand{\BB}{\mathbb{B}}
\newcommand{\NN}{\mathbb{N}}
\newcommand{\rl}{\mathbb R}
\newcommand{\rlnneg}{\mathbb R_{\geq0}}

\newcommand{\rln}{{\mathbb R}^{n}}
\newcommand{\rlm}{{\mathbb R}^{m}}
\newcommand{\rlp}{{\mathbb R}^{p}}
\newcommand{\rlpn}{{\mathbb R}^{p\times n}}
\newcommand{\rlnn}{{\mathbb R}^{n \times n}}

\newcommand{\on}{{\mathbbm 1}_n}
\newcommand{\om}{{\mathbbm 1}_{m}}
\newcommand{\until}[1]{\{1,\dots,#1\}}

\newcommand{\GG}{\mathcal{G}}
\newcommand{\LL}{\mathcal{L}}
\newcommand{\VV}{\mathcal{V}}
\newcommand{\AC}{\mathcal{AC}}
\newcommand{\CC}{\mathcal{C}}
\newcommand{\CCC}{\mathcal{C}^{1,1}}
\newcommand{\EE}{\mathcal{E}}

\newcommand{\marginJC}[1]{\marginpar{\color{red}\tiny\ttfamily#1}}
\newcommand{\margins}[1]{\marginpar{\color{blue}\tiny\ttfamily#1}}
\renewcommand{\marginJC}[1]{}
\renewcommand{\margins}[1]{}

\newcommand{\oprocendsymbol}{\hbox{$\bullet$}}
\newcommand{\oprocend}{\relax\ifmmode\else\unskip\hfill\fi\oprocendsymbol}

\newcommand{\myclearpage}{\clearpage}
\renewcommand{\myclearpage}{}

\title{%
  Distributed Coordination for Nonsmooth Convex \\ Optimization via
  Saddle-Point Dynamics\thanks{Submitted to the SIAM Journal on
    Control and Optimization on June 29, 2016, revised version
    on~\today. This work was supported by the German National Academic
    Foundation, the Dr.~J{\"u}rgen and Irmgard Ulderup Foundation,
    Award FA9550-15-1-0108, and NSF Award CNS-1329619. Preliminary
    versions of this manuscript appeared
    as~\cite{SKN-FA-JC:16-cdc,SKN-JC:15-cdc}~at the IEEE Conference on
    Decision and Control.}  }

\author{%
	Simon K. Niederl{\"a}nder\thanks{Institute for Systems Theory and Automatic Control, University of Stuttgart, 70550 Stuttgart, Germany, Email: {\tt\small simon.niederlaender@ist.uni-stuttgart.de}.}
\and
	Jorge Cort\'{e}s\thanks{Department of Mechanical and Aerospace Engineering, Jacobs School of Engineering, University of California, San Diego, 9500 Gilman Drive, La Jolla, CA 92093, United States, Email: {\tt\small cortes@ucsd.edu}.}
}

\begin{document}
\maketitle
\slugger{sicon}{2015}{xx}{x}{x--x}

\begin{abstract}
  This paper considers continuous-time coordination algorithms for
  networks of agents that seek to collectively solve a general class
  of nonsmooth convex optimization problems with an inherent
  distributed structure. Our algorithm design builds on the
  characterization of the solutions of the nonsmooth convex program as
  saddle points of an augmented Lagrangian. We show that the
  associated saddle-point dynamics are asymptotically correct but, in
  general, not distributed because of the presence of a global penalty
  parameter. This motivates the design of a discontinuous
  saddle-point-like algorithm that enjoys the same convergence
  properties and is fully amenable to distributed implementation. Our
  convergence proofs rely on the identification of a novel global
  Lyapunov function for saddle-point dynamics. This novelty also
  allows us to identify mild convexity and regularity conditions on
  the objective function that guarantee the exponential convergence
  rate of the proposed algorithms for convex optimization problems
  subject to equality constraints.  Various examples illustrate our
  discussion.
\end{abstract}




\pagestyle{myheadings}
\thispagestyle{plain}
\markboth{S. K. NIEDERL\"{A}NDER AND J. CORT\'{E}S}{DISTRIBUTED
	COORDINATION FOR NONSMOOTH CONVEX OPTIMIZATION}

\section{Introduction}\label{sec1}

Distributed convex optimization problems arise in a wide range of
scenarios involving multi-agent systems, including network flow
optimization, control of distributed energy resources, resource
allocation and scheduling, and multi-sensor fusion. In such contexts,
the goals and performance metrics of the agents are encoded into
suitable objective functions whose optimization may be subject to a
combination of physical, communication, and operational
constraints. Decentralized algorithmic approaches to solve these
optimization problems yield various advantages over centralized
solvers, including reduced communication and computational overhead at
a single point via spatially distributed processors, robustness
against malfunctions, or the ability to quickly react to changes. In
this paper we are motivated by network scenarios that give rise to
general nonsmooth convex optimization problems with an intrinsic
distributed nature. We consider convex programs with an additively
separable objective function and local coupling equality and
inequality constraints. Our objective is to synthesize distributed
coordination algorithms that allow each agent to find their own
component of the optimal solution vector. This setup substantially
differs from consensus-based distributed optimization where agents
agree on the entire optimal solution vector. We also seek to provide
algorithm performance guarantees by way of characterizing the
convergence rate of the network state towards the optimal solution. We
see these characterizations as a stepping stone towards the
development of strategies that are robust against disturbances and can
accommodate a variety of resource constraints.

\subsection*{Literature Review}\label{sec1:sub1}

The interest on networked systems has stimulated the synthesis of
distributed strategies that have agents interacting with neighbors to
coordinate their computations and solve convex optimization problems
with constraints~\cite{DPB:99,DPB-JNT:97}. A majority of works focus
on consensus-based approaches, where individual agents maintain,
communicate, and update an estimate of the entire solution vector of
the optimization problem, implemented in discrete time, see
e.g.,~\cite{JCD-AA-MJW:12,BJ-MR-MJ:09,AN-AO:09,AN-AO-PAP:10,MZ-SM:12}
and references therein. Recent
work~\cite{BG-JC:14-tac,JL-CYT:12,DMN-JC:16-sicon,JW-NE:11} has
proposed a number of continuous-time solvers whose convergence
properties can be studied using notions and tools from classical
stability analysis tools.  This continuous-time framework facilitates
the explicit computation of the evolution of candidate Lyapunov
functions and their Lie derivatives, opening the way to a systematic
characterization of additional desirable algorithm properties such as
speed of convergence, disturbance rejection, and robustness to
uncertainty.  In contrast to consensus-based approaches, and of
particular importance to our work here, are distributed strategies
where each agent seeks to determine only its component of the optimal
solution vector (instead of the whole one) and interchanges
information with neighbors whose size is independent of the
networks'. Such strategies are particularly well suited for convex
optimization problems over networks that involve an aggregate
objective function that does not couple the agents' decisions but
local (equality or inequality) constraints that instead do. Dynamics
enjoying such scalability properties include the partition-based dual
decomposition algorithm for network optimization proposed
in~\cite{RC-GN:13}, the discrete-time algorithm for non-strict convex
problems in~\cite{IN-JS:08} that requires at least one of the exact
solutions of a local optimization problem at each iteration, and the
inexact algorithm in~\cite{IN-VN:14} that only achieves convergence to
an approximate solution of the optimization problem. In the context of
neural networks, the work~\cite{MF-PN-MQ:04} proposes a generalized
circuit for nonsmooth nonlinear optimization based on first-order
optimality conditions with convergence guarantees. However, the
proposed dynamics are not fully amenable to distributed implementation
due to the global penalty parameters involved. A common approach to
design such distributed strategies relies on the saddle-point or
primal-dual dynamics~\cite{KA-LH-HU:58,TK:56,BTP:70} corresponding to
the Lagrangian of the optimization problem. The work~\cite{DF-FP:10}
studies primal-dual gradient dynamics for convex programs subject to
inequality constraints. These dynamics are modified with a projection
operator on the dual variables to preserve their
nonnegativity. Although convergence in the primal variables is
established, the dual variables converge to some unknown point which
might not correspond to a dual solution. The work~\cite{DR-JC:15-tac}
introduces set-valued and discontinuous saddle-point algorithms
specifically tailored for linear programs. More recently, the
work~\cite{AC-BG-JC:17-sicon} studies the asymptotic convergence
properties of the saddle-point dynamics associated to general saddle
functions. Our present work contributes to this body of literature on
distributed algorithms based on saddle-point dynamics, with the key
distinctions of the generality of the problem considered and the fact
that our technical analysis relies on Lyapunov, rather than LaSalle,
arguments to establish asymptotic convergence and performance
guarantees. Another distinguishing feature of the present work is the
explicit characterization of the exponential convergence rate of
continuous-time coordination algorithms for convex optimization
problems subject to equality constraints.

\subsection*{Statement of Contributions}\label{sec1:sub2}

We consider generic nonsmooth convex optimization problems defined by
an additively separable objective function and local coupling
constraints. Our starting point is the characterization of the
primal-dual solutions of the nonsmooth convex program as saddle points
of an augmented Lagrangian which incorporates quadratic regularization
and $\ell_1$-exact-penalty terms to eliminate the inequality
constraints. This problem reformulation motivates the study of the
saddle-point dynamics (gradient descent in the primal variable and
gradient ascent in the dual variable) associated with the augmented
Lagrangian. Our first contribution is the identification of a novel
nonsmooth Lyapunov function which allows us to establish the
asymptotic correctness of the algorithm without relying on arguments
based on the LaSalle invariance principle. Our second contribution
pertains the performance characterization of the proposed coordination
algorithms. We restrict our study to the case when the convex
optimization problem is subject to equality constraints only. For this
scenario, we rely on the Lyapunov function identified in the
convergence analysis to provide sufficient conditions on the objective
function of the convex program to establish the exponential
convergence of the algorithm and characterize the corresponding
rate. Since the proposed saddle-point algorithm relies on a priori
global knowledge of a penalty parameter associated with the
exact-penalty terms introduced to ensure convergence to the solutions
of the optimization problem, our third contribution is an alternative,
discontinuous saddle-point-like algorithm that does not require such
knowledge and is fully amenable to distributed implementation over a
group of agents. We show that, given any solution of the
saddle-point-like algorithm, there exists a value of the penalty
parameter such that the trajectory is also a solution of the
saddle-point dynamics, thereby establishing that both dynamics enjoy
the same convergence properties. As an additional feature, the
proposed algorithm guarantees feasibility with respect to the
inequality constraints for any time.  Various examples illustrate our
discussion.

\subsection*{Organization}\label{sec1:sub3}

The paper is organized as follows. Section~\ref{sec2}~introduces basic
notions on nonsmooth analysis and set-valued dynamical
systems. Section~\ref{sec3} proposes the saddle-point algorithm to
solve the problem of interest, establishes its asymptotic correctness
and characterizes the exponential converge rate. Section~\ref{sec4}
proposes the saddle-point-like algorithm and studies its relation to
the saddle-point dynamics. Section~\ref{sec5} discusses the
distributed implementation of the algorithms and illustrates the
results through various examples. Finally, Section~\ref{sec6}
summarizes our conclusions and ideas for future work.

\myclearpage
\section{Preliminaries}\label{sec2}

We let $\langle\cdot,\cdot\rangle$ denote the Euclidean inner product
and $\lVert\cdot\rVert$, respectively $\lVert\cdot\rVert_{\infty}$,
denote the $\ell_2$- and $\ell_\infty$-norms in $\rln$. The Euclidean
distance from a point $x\in\rln$ to a set $X\subset\rln$ is denoted by
$\dist(x,X)$. Let $\on=(1,\dots,1)\in\rln$. Given $x\in\rln$, let
$[x]^{+}=(\max\{0,x_{1}\},\dots,\max\{0,x_{n}\})\in\rln$. Given a set
$X\subset\rln$, we denote its convex hull by $\co X$, its interior by
$\inter X$, and its boundary by $\bd X$. The closure of $X$ is denoted
by $\cl X=\inter X\cup\bd X$. Let $\BB(x,\delta)=\{y\in\rln\mid\lVert
y-x\rVert<\delta\}$ and $\overline{\BB}(x,\delta)=\{y\in\rln\mid\lVert
y-x\rVert\leq\delta\}$ be the open and closed ball, centered at
$x\in\rln$ of radius $\delta>0$. Given $X,Y\subset\rln$, the Minkovski
sum of $X$ and $Y$ is defined by $X+Y=\{x+y\mid x\in X,\ y\in Y\}$.

A function $f:\rln\to\rl$ is~\emph{convex} if $f(\theta
x+(1-\theta)y)\leq\theta f(x)+(1-\theta)f(y)$ for all $x,y\in\rln$ and
$\theta\in[0,1]$. $f$ is~\emph{strictly convex} if $f(\theta
x+(1-\theta)y)<\theta f(x)+(1-\theta)f(y)$ for all $x\neq y$ and
$\theta\in(0,1)$. A function $f:\rln\to\rl$ is~\emph{positive
  definite} with respect to $X\subset\rln$ if $f(x)=0$ for all $x\in
X$ and $f(x)>0$ for all $x\notin X$. We say $f:\rln\to\rl$
is~\emph{coercive} with respect to $X\subset \rln$ if
$f(x)\rightarrow+\infty$ when $\dist(x,X)\rightarrow+\infty$. Given
$\gamma>0$, the~\emph{$\gamma$-sublevel set} of $f$ is
$\lev(f)=\{x\in\rln\mid f(x)\leq\gamma\}$. A bivariate function
$f:\rln\times\rlp\to\rl$ is~\emph{convex-concave} if it is convex in
its first argument and concave in its second. A~\emph{set-valued map}
$F:\rln\rightrightarrows\rln$ maps elements of $\rln$ to elements of
$2^{\rln}$. A set-valued map $F:\rln\rightrightarrows\rln$
is~\emph{monotone} if $\langle x-y,\xi_{x}-\xi_{y}\rangle\geq0$
whenever $\xi_{x}\in F(x)$ and $\xi_{y}\in F(y)$. Finally, $F$
is~\emph{strictly monotone}~if $\langle x-y,\xi_{x}-\xi_{y}\rangle>0$
whenever $\xi_{x}\in F(x)$, $\xi_{y}\in F(y)$ and $x\neq y$.

\subsection{Nonsmooth Analysis}\label{sec2:sub1}

We review here relevant basic notions from nonsmooth
analysis~\cite{FHC:83} that will be most helpful in both our algorithm
design and analysis. A function $f:\rln\to\rl$ is~\emph{locally
  Lipschitzian at} $x\in\rln$ if there exist $\delta_{x}>0$ and
$L_{x}>0$ such that $\lvert f(y)-f(z)\rvert\leq L_{x}\lVert y-z\rVert$
for all $y,z\in\BB(x,\delta_{x})$. The function $f$ is~\emph{locally
  Lipschitzian} if it is locally Lipschitzian at $x$, for all
$x\in\rln$. A convex function is locally Lipschitzian
(cf.~\cite[Theorem 3.1.1, p. 16]{JBHU-CL:93}).

Rademacher's Theorem~\cite{FHC:83} states that locally Lipschitzian
functions are continuously differentiable almost everywhere (in the
sense of Lebesgue measure). Let $\Omega_{f}\subset\rln$ be the set of
points at which $f$ fails to be differentiable, and let $S$ denote any
other set of measure zero. The~\emph{generalized gradient} $\partial
f:\rln\rightrightarrows\rln$ of $f$ at $x\in\rln$ is defined by
\begin{equation*}
  \partial f(x)=\co\Big\{\lim_{i\to+\infty}\nabla f(x_{i})\mid x_{i}\to x,\
  x_{i}\notin S\cup\Omega_{f}\Big\}.
\end{equation*}
Note that if $f$ is continuously differentiable at $x\in\rln$, then
$\partial f(x)$ reduces to the singleton set $\{\nabla f(x)\}$. If $f$
is convex, then $\partial f(x)$ coincides with
the~\emph{subdifferential} (in the sense of convex analysis), that is,
the set of subgradients $\xi\in\rln$ satisfying $f(y)\geq
f(x)+\langle\xi,y-x\rangle$ for all $y\in\rln$ (cf.~\cite[Proposition
2.2.7]{FHC:83}). With this characterization, it is not difficult to
see that $f$ is (strictly) convex if and only if $\partial f$ is
(strictly) monotone.

A set-valued map $F$ is~\emph{upper semi-continuous} if, for all
$x\in\rln$ and $\varepsilon>0$, there exists $\delta>0$ such that
$F(y)\subset F(x)+\BB(0,\varepsilon)$ for all $y\in\BB(x,\delta)$. We
say $F$ is~\emph{locally bounded} if, for every $x\in\rln$, there
exist $\varepsilon>0$ and $\delta>0$ such that
$\lVert\xi\rVert\leq\varepsilon$ for all $\xi\in F(y)$ and all
$y\in\BB(x,\delta)$. The following result summarizes some important
properties of the generalized gradient~\cite{FHC:83}.

\begin{proposition}[Properties of the generalized gradient]\label{sec2:sub1:pr:genprop}
  Let $f:\rln\to\rl$ be locally Lipschitzian at $x\in\rln$. Then,
  \begin{enumerate}[(i)]
  \item $\partial f(x)\subset\rln$ is nonempty, convex and compact,
    and $\lVert\xi\rVert\leq L_{x}$, for all $\xi\in\partial f(x)$;
  \item $\partial f(x)$ is upper semi-continuous at $x\in\rln$.
  \end{enumerate}
\end{proposition}

Let $\CCC(\rln,\rl)$ denote the class of functions $f:\rln\to\rl$ that
are continuously differentiable and whose gradient $\nabla
f:\rln\to\rln$ is locally Lipschitzian. The~\emph{generalized Hessian}
$\partial(\nabla f):\rln\rightrightarrows\rlnn$ of $f$ at $x\in\rln$
is defined by
\begin{equation*}
  \partial(\nabla f)(x)=\co\Big\{\lim_{i\to+\infty}\nabla^{2}f(x_{i})\mid x_{i}
  \to x,\ x_{i}\notin\Omega_{f}\Big\}.
\end{equation*}
By construction, $\partial(\nabla f)(x)$ is a nonempty, convex and
compact set of symmetric matrices which reduces to the singleton set
$\{\nabla^{2}f(x)\}$ whenever $f$ is twice continuously differentiable
at $x\in\rln$~\cite{JBHU-JJS-VHN:84}. The following result is a direct
extension of Lebourg's Mean-Value Theorem to vector-valued
functions~\cite{RTR-RJBW:98}.

\begin{proposition}[Extended Mean-Value
  Theorem]\label{sec2:sub1:pr:emvt}
  Let $\nabla f:\rln\to\rln$ be locally Lipschitzian and let
  $x,y\in\rln$. Then,
  \begin{equation*}
    \nabla f(x)-\nabla f(y)\in\co\big\{\partial(\nabla f([x,y]))\big\}(x-y),
  \end{equation*}
  where $\co\big\{\partial(\nabla f([x,y]))\big\}(x-y)=\co\{H(x-y)\mid
  H\in\partial(\nabla f)(z)\ {\rm for~some}\ z\in[x,y]\}$, and
  $[x,y]=\{x+\theta(y-x)\mid\theta\in[0,1]\}$.
\end{proposition}

\subsection{Set-Valued Dynamical Systems}\label{sec2:sub2}

Throughout the manuscript, we consider set-valued and locally
projected dynamical systems~\cite{JC:08-csm-yo,AN-DZ:96} defined by
differential inclusions~\cite{JPA-AC:84}. Let $X\subset\rln$ be open,
and let $F:X\rightrightarrows\rln$ be a set-valued map. Consider the
differential inclusion
\begin{equation}
  \renewcommand{\theequation}{DI}\tag{\theequation}\label{sec2:sub2:diffinclusion}
  \begin{cases}
    \dot{x}\in F(x) , 
    \\
    x(0)=x_{0}\in X.
  \end{cases}
\end{equation}
A solution of~\eqref{sec2:sub2:diffinclusion} on the interval
$[0,t_{+})\subset\rl$ (if any) is an absolutely continuous mapping
taking values in $X$, denoted by $x\in\AC([0,t_{+}),X)$, such that
$\dot{x}(t)\in F(x(t))$ for almost all $t\in[0,t_{+})$.  A point $x$
is an equilibrium of~\eqref{sec2:sub2:diffinclusion} if $0 \in
F(x)$. We denote by $\eq(F) $ the set of equilibria.  Given $x_{0}\in
X$, the existence of solutions of~\eqref{sec2:sub2:diffinclusion} with
initial condition $x_{0}\in X$ is guaranteed by the following
result~\cite{JC:08-csm-yo}.

\begin{lemma}[Existence of local solutions]\label{sec2:sub2:lm:exsol}
  Let the set-valued map $F:X\rightrightarrows\rln$ be locally
  bounded, upper semi-continuous with nonempty, convex and compact
  values. Then, given any $x_{0}\in X$, there exists a solution
  of~\eqref{sec2:sub2:diffinclusion} with
  initial condition $x_{0}$.
\end{lemma}

Given a locally Lipschitzian function $V:X\to\rl$,
the~\emph{set-valued Lie derivative} $\LL_{F}V:X\rightrightarrows\rl$
of $V$ with respect to $F$ at $x\in X$ is defined by
\begin{equation*}
  \LL_{F}V(x)=\big\{\psi\in\rl\mid\exists\xi\in F(x):\langle\xi,\pi\rangle=
  \psi,\ \forall\pi\in\partial V(x)\big\}.
\end{equation*}
For each $x\in X$, $\LL_{F}V(x)$ is a closed and bounded interval in
$\rl$, possibly empty.

Let $G\subset\rln$ be a nonempty, closed and convex
set. The~\emph{tangent cone} and the~\emph{normal cone} of $G$ at
$x\in G$ are, respectively,
\begin{equation*}
  T_{G}(x)=\cl\rlnneg(G-x),\qquad N_{G}(x)=\{n\in\rln\mid\langle
  n,y-x\rangle\leq0,\ \forall y\in G\}. 
\end{equation*}
%
%
Note that if $x\in\inter G$, then $T_{G}(x)=\rln$ and
$N_{G}(x)=\{0\}$. Let $\proj_{G}(x)=\argmin_{y\in G}\lVert
x-y\rVert$. The~\emph{orthogonal (set) projection}~of a nonempty,
convex and compact set $F(x)\subset\rln$ at $x\in G$ with respect to
$G\subset\rln$ is defined by
\begin{equation}\label{sec2:sub2:projop}
  P_{T_G(x)}(F(x))=\bigcup_{\xi\in F(x)}\lim_{\delta\searrow 0}
  \frac{\proj_{G}(x+\delta\xi)-x}{\delta}.
\end{equation}
Note that if $x\in\inter G$, then $P_{T_G(x)}(F(x))$ reduces to the
set $F(x)$. By definition, the orthogonal projection
$P_{T_G(x)}(F(x))$ is equivalent to the Euclidean projection of
$F(x)\subset\rln$ onto the tangent cone $T_{G}(x)$ at $x\in G$, i.e.,
$P_{T_G(x)}(F(x))=\proj_{T_{G}(x)}(F(x))$, cf.~\cite[Remark
1.1]{MP-MP:02}. Consider now the locally projected differential
inclusion
\begin{equation}
  \renewcommand{\theequation}{PDI}\tag{\theequation}\label{sec2:sub2:projinclusion}
  \begin{cases}
    \dot{x}\in P_{T_G(x)}(F(x)),
    \\
    x(0)=x_{0}\in G.
  \end{cases}
\end{equation}
Note that, in general, the set-valued map $x\mapsto P_{T_G(x)}(F(x))$
possesses no continuity properties and the values of
$P_{T_G(x)}(F(x))$ are not necessarily convex~\cite{JPA-AC:84}. Still,
the following result states conditions under which solutions
of~\eqref{sec2:sub2:projinclusion} exist~\cite{CH:73}.

\begin{lemma}[Existence of local solutions of projected differential
  inclusions]\label{sec2:sub2:lm:prexistence}
  Let $G\subset\rln$ be nonempty, closed and convex, and let the
  set-valued map $F:G\rightrightarrows\rln$ be locally bounded, upper
  semi-continuous with nonempty, convex and compact values. If there
  exists $c>0$ such that, for every $x\in G$,
  \begin{equation*}
    \sup_{\xi\in F(x)}\lVert\xi\rVert\leq c(1+\lVert x\rVert),
  \end{equation*}
  then, for any $x_{0}\in G$, there exists at least one solution
  $x\in\AC([0,t_{+}),G)$ of~\eqref{sec2:sub2:projinclusion} with
  initial condition $x_{0}$.
\end{lemma}

\myclearpage
\section{Convex Optimization via Saddle-Point Dynamics}\label{sec3}

Consider the constrained minimization problem
\begin{equation}
  \renewcommand{\theequation}{P}\tag{\theequation}\label{sec3:probstatement}
  \min\{f(x)\mid h(x)=0_{p},\ g(x)\leq0_{m}\},
\end{equation}
where $f:\rln\to\rl$ and $g:\rln\to\rlm$ are convex, and
$h:\rln\to\rlp$ is affine, i.e., $h(x)=Ax-b$, with $A\in\rlpn$ and
$b\in\rlp$, where $p\leq n$. Let $C=\{x\in\rln\mid h(x)=0_{p},\
g(x)\leq0_{m}\}$ denote the constraint set and assume that the (closed
and convex) set of solutions $S=\{x^{\star}\in C\mid
f(x^{\star})=\inf_{C}f\}$ of~\eqref{sec3:probstatement} is nonempty
and bounded. Throughout the paper we assume that the constraint set
$C\subset\rln$ satisfies the strong Slater
assumptions~\cite{JBHU-CL:93}, i.e.,
\begin{enumerate}[\hspace{10pt}({A}1)]
\item $\rank(A)=p$, i.e., the rows of $A\in\rlpn$ are linearly
  independent;
\item $\exists x\in\rln$ such that $h(x)=0_{p}$ and $g_{k}(x)<0$ for
  all $k\in\until{m}$.
\end{enumerate}

Our main objective is to design continuous-time algorithms with
performance guarantees to find the solution of the nonsmooth convex
program~\eqref{sec3:probstatement}. We are specifically interested in
solvers that are amenable to distributed implementation by a group of
agents, permitting each one of them to find their component of the
solution vector.  The algorithms proposed in this work build on
concepts of Lagrangian duality theory and characterize the primal-dual
solutions of~\eqref{sec3:probstatement}~as saddle points of an
augmented Lagrangian. More precisely, let $\kappa,\mu>0$ and let the
augmented Lagrangian $L:\rln\times\rlp\to\rl$ associated
with~\eqref{sec3:probstatement} be defined by
\begin{equation*}
  L(x,\lambda)=f(x)+\frac{1}{2\mu}\lVert h(x)\rVert^{2}+\langle\lambda,h(x)
  \rangle+\kappa\langle\om,[g(x)]^{+}\rangle,
\end{equation*}
where $\lambda\in\rlp$ is a Lagrange multiplier. Under the regularity
assumptions (A1)--(A2), for every $x^{\star}\in S$, there exists
$(\lambda^{\star},\nu^{\star})\in\rlp\times\rlm$ such that
$(x^{\star},\lambda^{\star})$ is a~\emph{saddle point} of the
augmented Lagrangian $L$, i.e.,
\begin{equation*}
  L(x^{\star},\lambda)\leq L(x^{\star},\lambda^{\star})\leq
  L(x,\lambda^{\star}),\quad\forall x\in\rln,\ \forall\lambda\in\rlp,
\end{equation*}
provided that $\kappa\geq\lVert\nu^{\star}\rVert_{\infty}$. We denote
the (closed and convex) set of saddle points of $L$ by $\sap(L)$. The
following statement reveals that the converse result is also true. Its
proof can be deduced from results on penalty functions in the
optimization literature, see e.g.~\cite{DPB:75b}, and we omit it for
reasons of space.

\begin{lemma}[Saddle Points and Solutions of the Optimization
  Problem]\label{sec3:lm:spthm}
  Let $L:\rln\times\rlp\to\rl$ and let
  $(x^{\star},\lambda^{\star})\in\sap(L)$ with
  $\kappa>\lVert\nu^{\star}\rVert_{\infty}$ for some dual solution
  $\nu^{\star}$ of~\eqref{sec3:probstatement}. Then, $x^{\star}\in S$,
  i.e., $x^{\star}$ is a (primal) solution
  of~\eqref{sec3:probstatement}.
\end{lemma}

Lemma~\ref{sec3:lm:spthm} identifies a condition under which the
penalty parameter $\kappa$ is~\emph{exact}~\cite{DPB:75b,OLM:85}. This
condition in turn depends on the dual solution set.  Given this
result, instead of directly solving~\eqref{sec3:probstatement}, we
seek to design strategies that find saddle points of $L$. Since the
bivariate augmented Lagrangian $L$ is, by definition, convex-concave,
a natural approach to find the saddle points is via its associated
saddle-point dynamics
\begin{equation}
  \renewcommand{\theequation}{SPD}\tag{\theequation}\label{sec3:spdynamics}
  \begin{cases}
    (\dot{x},\dot{\lambda})\in -(\partial_{x}L,-\partial_{\lambda}L)
    (x,\lambda) ,
    \\
    (x(0),\lambda(0))=(x_{0},\lambda_{0})\in\rln\times\rlp.
  \end{cases}
\end{equation}
Let $(\partial_{x}L,-\partial_{\lambda}L):
\rln\times\rlp\rightrightarrows\rln\times\rlp$ denote
the~\emph{saddle-point operator} associated with $L$. Since $L$ is
convex-concave and locally Lipschitzian~\cite[Theorem 35.1]{RTR:70},
the mapping
$(x,\lambda)\mapsto(\partial_{x}L,-\partial_{\lambda}L)(x,\lambda)$ is
locally bounded, upper semi-continuous and takes nonempty, convex and
compact values, cf. Proposition~\ref{sec2:sub1:pr:genprop}. The
existence of local solutions $(x,\lambda)\in\AC([0,t_{+}),\rln\times\rlp)$
of~\eqref{sec3:spdynamics} is guaranteed by
Lemma~\ref{sec2:sub2:lm:exsol}. Moreover, we have that
\begin{equation*}
  \sap(L)=(\partial_{x}L,-\partial_{\lambda}L)^{-1}(0_{n},0_{p}) = \eq(\partial_{x}L,-\partial_{\lambda}L).
\end{equation*}
Consequently, our strategy to solve~\eqref{sec3:probstatement} amounts
to the issue of ``finding the zeros'' of the saddle-point operator
via~\eqref{sec3:spdynamics}.

\begin{remark}[Existence and uniqueness of solutions]
  {\rm In fact, one can show that the saddle-point operator
    $(\partial_{x}L,-\partial_{\lambda}L):\rln\times\rlp\rightrightarrows\rln\times\rlp$
    is maximal monotone, and thus, the existence and uniqueness of a
    global solution of~\eqref{sec3:spdynamics} 
    follows from~\cite[Theorem 1, p. 147]{JPA-AC:84}.}\oprocend
\end{remark}

\subsection{Convergence Analysis}\label{sec3:sub1}

By construction of the saddle-point dynamics~\eqref{sec3:spdynamics},
it is natural to expect that its trajectories converge towards the set
of saddle points of the augmented Lagrangian $L$ as time evolves. Our
proof strategy to establish this convergence result relies on
Lyapunov's direct method.

\begin{theorem}[Asymptotic convergence]\label{sec3:sub1:th:lyapunov}
  Let $L:\rln\times\rlp\to\rl$ with $\mu\in(0,1)$ and
  $\kappa>\lVert\nu^{\star}\rVert_{\infty}$ for some dual solution
  $\nu^{\star}$ of~\eqref{sec3:probstatement}. Then, the set $\sap(L)$
  is strongly globally asymptotically stable
  under~\eqref{sec3:spdynamics}.
\end{theorem}
\begin{proof}
  We start by observing that the set of saddle points $\sap(L)$ is
  nonempty, convex and compact given that
  $\kappa>\lVert\nu^{\star}\rVert_{\infty}$, the solution set
  $S\subset\rln$ of~\eqref{sec3:probstatement} is nonempty and
  bounded, and that the strong Slater assumptions hold,
  cf.~\cite[Theorem 2.3.2]{JBHU-CL:93}.  Consider the Lyapunov
  function candidate $V:\rln\times\rlp\to\rl$ defined by
  \begin{equation}\begin{split}\label{sec3:sub1:lyapunov}
      V(x,\lambda)=&~f(x)-\inf\nolimits_{C}f+\frac{1}{2\mu}\lVert
      h(x)\rVert^{2}
      +\langle\lambda,h(x)\rangle+\kappa\langle\om,[g(x)]^{+}\rangle
      \\
      &+\frac{1}{2}\dist((x,\lambda),\sap(L))^{2}.
    \end{split}\end{equation}
  
  We start by showing that $V(x,\lambda)>0$ for all
  $(x,\lambda)\notin\sap(L)$ and $V(x,\lambda)=0$ if
  $(x,\lambda)\in\sap(L)$. Let $(x,\lambda)\in\rln\times\rlp$ and let
  \begin{align}\label{sec3:sub1:mindist}
    (x^{\star}(x,\lambda),\lambda^{\star}(x,\lambda))=\argmin_{(\tilde{x},
      \tilde{\lambda})\in\sap(L)}\Big(\frac{1}{2}\lVert x-\tilde{x}\rVert^{2}+
    \frac{1}{2}\lVert\lambda-\tilde{\lambda}\rVert^{2}\Big).
  \end{align}
  For notional convenience, we drop the dependency of
  $(x^{\star},\lambda^{\star})$ on the argument $(x,\lambda)$. Note
  that the (unique) minimizer $(x^{\star},\lambda^{\star})$
  of~\eqref{sec3:sub1:mindist} exists since the set $\sap(L)$ is
  nonempty, 
  convex and compact. By convexity of $f$, $h$ and
  $\langle\om,[g]^{+}\rangle$,
  \marginJC{The wording ``by convexity of...'' hides a lot of the
  mathematical derivations that need to be done to get the
  inequality. This makes it much harder for the reader to follow our
  the reasoning, worsening the overall experience. I'd suggest to
  elaborate on this, here and later at some other points, along the
  lines we had before.}
\margins{I would prefer to tell the reader the ``main ingredients''
  used rather than spamming him/her with technicalities such as
  rearranging terms.}
  \begin{align*}
    V(x,\lambda)\geq&~\frac{1}{2\mu}\lVert h(x)\rVert^{2}+\langle\lambda-
    \lambda^{\star},h(x)\rangle+\frac{1}{2}\lVert x-x^{\star}\rVert^{2}+
    \frac{1}{2}\lVert\lambda-\lambda^{\star}\rVert^{2}\\
    =&~\frac{1}{2}\big\langle P(x-x^{\star},\lambda-\lambda^{\star}),
    (x-x^{\star},\lambda-\lambda^{\star})\big\rangle
  \end{align*}
  where the symmetric matrix
  \begin{align}\label{sec3:matrixp}
    P=\begin{pmatrix}
      \dfrac{1}{\mu} A^{\top}A+I_{n}&A^{\top}\\A&I_{p}
    \end{pmatrix}\in\rl^{(n+p)\times(n+p)}
  \end{align}
  is positive definite for $\mu\in(0,1]$ (this follows by observing
  that $I_p \succ 0$ and the Schur complement~\cite{RAH-CRJ:85}~of
  $I_p$ in $P$, denoted $P/I_p = \frac{1}{\mu}A^\top A + I_n - A^\top
  I_p^{-1} A$, is positive definite).  Since
  $(x,\lambda)\in\rln\times\rlp$ is arbitrary, it follows
  \begin{align*}
    V(x,\lambda)\geq\frac{1}{2}\lambda_{\min}(P)\dist((x,\lambda),\sap(L))^{2}.
  \end{align*}
  Thus, we have $V(x,\lambda)>0$ for all $(x,\lambda)\notin\sap(L)$
  and $V(x,\lambda)=0$ if $(x,\lambda)\in\sap(L)$. Moreover, $V$ is
  coercive with respect to $\sap(L)$, that is,
  $V(x,\lambda)\to+\infty$ whenever
  $\dist((x,\lambda),\sap(L))\to+\infty$.

  We continue by studying the evolution of $V$ along the solutions of
  the saddle-point dynamics~\eqref{sec3:spdynamics}. Let
  $(x,\lambda)\in\rln\times\rlp$ and let $(x^{\star},\lambda^{\star})$
  be defined by~\eqref{sec3:sub1:mindist}. Take $\psi\in\LL_{\rm
    SPD}V(x,\lambda)$. By definition of the set-valued Lie derivative,
  there exists
  $(\pi_{x},\pi_{\lambda})\in(\partial_{x}L,-\partial_{\lambda}L)(x,\lambda)$
  such that
  $\psi=-\langle(\xi_{x},\xi_{\lambda}),(\pi_{x},\pi_{\lambda})\rangle$
  for all
  $(\xi_{x},\xi_{\lambda})\in(\partial_{x}V,\partial_{\lambda}V)(x,\lambda)$. From
  this and the convexity of $f$ and $\langle\om,[g]^{+}\rangle$, we
  obtain
  %
  %
  \begin{equation}\begin{split}\label{sec3:sub1:psi}
      \psi\leq&-\Big\lVert\frac{1}{\mu}A^{\top}h(x)+A^{\top}\lambda+\xi_f+
      \kappa~\sum_{\mathclap{k\in\until{m}}}~\xi_{g_{k}}^{+}\Big\rVert^{2}-
      \Big(\frac{1}{\mu}-1\Big)\lVert-h(x)\rVert^{2}\\
      \leq&-\min\Big\{1,\frac{1}{\mu}-1\Big\}\Big\lVert\Big(\frac{1}{\mu}A^{\top}
      h(x)+A^{\top}\lambda+\xi_f +
      \kappa~\sum_{\mathclap{k\in\until{m}}}~\xi_{g_{k}}^{+},-h(x)\Big)
      \Big\rVert^{2},
    \end{split}
  \end{equation}
  where $\xi_{f}\in\partial f(x)$ and
  $\xi_{g_{k}}^{+}\in\partial[g_{k}(x)]^{+}$. Since $\mu\in(0,1)$, and
  using the fact that
  $\sap(L)=(\partial_{x}L,-\partial_{\lambda}L)^{-1}(0_{n},0_{p})$, we
  deduce that the right-hand-side of~\eqref{sec3:sub1:psi} equals to
  zero if and only if $(x,\lambda)\in\sap(L)$.  Since
  $(x,\lambda)\in\rln\times\rlp$ and $\psi\in\LL_{\rm
    SPD}V(x,\lambda)$ are arbitrary, it follows $\LL_{\rm
    SPD}V(x,\lambda)\subset(-\infty,0)$ for all
  $(x,\lambda)\notin\sap(L)$.  Hence, the set of saddle points
  $\sap(L)$ is strongly globally asymptotically stable
  under~\eqref{sec3:spdynamics}, concluding the proof.
\end{proof}

Recall that the set of saddle points $\sap(L)$ depends on the
parameter $\kappa$ (but not on $\mu$) and thus, trajectories of the
saddle-point dynamics~\eqref{sec3:spdynamics} need not converge to a
(primal) solution of~\eqref{sec3:probstatement}, unless the penalty
parameter $\kappa$ is exact, cf. Lemma~\ref{sec3:lm:spthm}.

\begin{remark}[Alternative convergence proof via the LaSalle
  Invariance Principle]\label{sec3:rm:lasalle}
  {\rm The strong global asymptotic stability of $\sap(L)$
    under~\eqref{sec3:spdynamics} can also be established using the
    alternative (weak) Lyapunov function
    \begin{equation*}
      V(x,\lambda)=\frac{1}{2}\lVert x-x^{\star}\rVert^{2}+\frac{1}{2}\lVert
      \lambda-\lambda^{\star}\rVert^{2},
    \end{equation*}
    where $(x^{\star},\lambda^{\star})\in\sap(L)$ is arbitrary. In
    fact, from the proof of Theorem~\ref{sec3:sub1:th:lyapunov}, one
    can deduce that the Lie derivative of $V$
    along~\eqref{sec3:spdynamics} is negative semidefinite, implying
    stability, albeit not asymptotic stability. To conclude the
    latter, one can invoke the LaSalle Invariance Principle for
    differential inclusions~\cite{AB-FC:99} to identify the limit
    points of the trajectories as the set of saddle points. In fact,
    this is the approach commonly taken in the literature
    characterizing the convergence properties of saddle-point
    dynamics, see
    e.g.,~\cite{AC-BG-JC:17-sicon,DF-FP:10,TH-IL:14,JW-NE:11} and
    references therein.  This approach has the disadvantage that, $V$
    not being a strict Lyapunov function, it cannot be used to
    characterize properties of the solutions
    of~\eqref{sec3:spdynamics} beyond asymptotic convergence. By
    constrast, the Lyapunov function~\eqref{sec3:sub1:lyapunov}
    identified in the proof of Theorem~\ref{sec3:sub1:th:lyapunov}
    opens the way to the study of other properties of the solutions
    such as the characterization of the rate of convergence, the
    robustness against disturbances via the notion of input-to-state
    stability, or the design of opportunistic state-triggered
    implementations that naturally result in aperiodic discrete-time
    algorithms.  }\oprocend
\end{remark}

Point-wise convergence of the solutions of~\eqref{sec3:spdynamics}~in
the set $\sap(L)$ follows from the stability of the each individual
saddle point and the asymptotic stability of the set $\sap(L)$
established in Theorem~\ref{sec3:sub1:th:lyapunov}, as stated in the
following result. The proof is analogous to the case of ordinary
differential equations, cf.~\cite[Corollary~5.2]{SPB-DSB:03}), and
hence we omit it for reasons of space.

\begin{corollary}[Point-wise asymptotic
  convergence]\label{sec3:sub1:co:pwconv}
  Any solution $(x,\lambda)\in\AC([0,+\infty),\rln\times\rlp)$
  of~\eqref{sec3:spdynamics}~starting from $\rln\times\rlp$ converges
  asymptotically to a point in the set~$\sap(L)$.
\end{corollary}

\subsection{Performance Characterization}\label{sec3:sub2}

In this section, we characterize the exponential convergence rate of
solutions of the saddle-point dynamics~\eqref{sec3:spdynamics} for the
case when the convex optimization problem~\eqref{sec3:probstatement}
is subject to equality constraints only. In order to do so, we pose
additional convexity and regularity assumptions on the objective
function of~\eqref{sec3:probstatement}. We have gathered in the
Appendix various intermediate results to ease the exposition of the
following result.

\begin{theorem}[Exponential convergence]\label{sec3:sub2:th:exp}
  Let $L:\rln\times\rlp\to\rl$ with $\mu\in(0,1)$. Suppose that
  $f\in\CCC(\rln,\rl)$ and $\partial(\nabla f)\succ0$. Then, the
  (singleton) set $\sap(L)$ is exponentially stable
  under~\eqref{sec3:spdynamics}.
\end{theorem}
\begin{proof}
  Under the assumptions of the result, note that the
  dynamics~\eqref{sec3:spdynamics} take the form of a differential equation,
\begin{equation*}
  \begin{cases}
    (\dot{x},\dot{\lambda}) = -(\nabla_{x}L,-\nabla_{\lambda}L)
    (x,\lambda) ,
    \\
    (x(0),\lambda(0))=(x_{0},\lambda_{0})\in\rln\times\rlp.
  \end{cases}
\end{equation*}
  Let $L:\rln\times\rlp\to\rl$ with $\mu\in(0,1)$. Since
  $\partial(\nabla f)\succ0$, it follows that $\nabla f$ is strictly
  monotone~\cite[Example 2.2]{JBHU-JJS-VHN:84}. Therefore, for any
  fixed $\lambda$, the mapping $x\mapsto \nabla_{x}L(x,\lambda)$ is
  strictly monotone as well. Thus, by assumption (A1), the set of
  saddle points of $L$ is a singleton, i.e.,
  $\sap(L)=\{x^{\star}\}\times\{\lambda^{\star}\}$, where
  $\lambda^{\star}=-(AA^{\top})^{-1}A\nabla f(x^{\star})$. In this
  case, the Lyapunov function $V$ as defined
  in~\eqref{sec3:sub1:lyapunov} reads
  \begin{equation*}
    V(x,\lambda)=f(x)-f(x^{\star})+\frac{1}{2\mu}\lVert h(x)\rVert^{2}
    +\langle\lambda,h(x)\rangle+\frac{1}{2}\dist((x,\lambda),\sap(L))^{2}.
  \end{equation*}
  and we readily obtain from the proof of
  Theorem~\ref{sec3:sub1:th:lyapunov} that
  \begin{equation*}
    V(x,\lambda)\geq\frac{1}{2}\lambda_{\min}(P)\dist((x,\lambda),\sap(L))^{2},
  \end{equation*}
  for all $(x,\lambda)\in\rln\times\rlp$, where
  $P\in\rl^{(n+p)\times(n+p)}$ is defined as in~\eqref{sec3:matrixp}.

  Our next objective is to upper bound the evolution of $V$ along the
  solutions of~\eqref{sec3:spdynamics} in terms of the distance to
  $\sap(L)$. Let $(x,\lambda)\in\rln\times\rlp$ and consider the Lie
  derivative of $V$ with respect to~\eqref{sec3:spdynamics} at
  $(x,\lambda)$, i.e.,
  \begin{align*}
    \LL_{\rm SPD}V(x,\lambda)=-\langle x-x^{\star}, \nabla f(x)-\nabla
    f(x^{\star})\rangle - \lVert\nabla_{x}L(x,\lambda)\rVert^{2} -
    \Big(\frac{1}{\mu}-1\Big)\lVert-\nabla_{\lambda}L(x,\lambda)\rVert^{2},
  \end{align*}
  where we have used the fact that $0 = \nabla_x L(x^\star,\lambda^\star)
  = \nabla f(x^\star) + \frac{1}{\mu} A^\top h(x^\star) + A^\top
  \lambda^{\star}$ and $0 = \nabla_\lambda L(x^\star,\lambda^\star) = h(x^\star)$.
  Now, since $f\in\CCC(\rln,\rl)$, the gradient $\nabla f$ is locally
  Lipschitzian and thus, the extended Mean-Value Theorem,
  cf. Proposition~\ref{sec2:sub1:pr:emvt}, yields
  \begin{equation*}
    \nabla f(x) - \nabla f(x^{\star})\in \co\big\{\partial(\nabla f([x,
    x^{\star}]))\big\}(x-x^{\star}),
  \end{equation*}
  where
  $[x,x^{\star}]=\{x+\theta(x^{\star}-x)\mid\theta\in[0,1]\}$. By
  Lemma~\ref{app:lm:posdef}, there exists a symmetric and positive
  definite matrix $H(x)\in\co\big\{\partial(\nabla
  f([x,x^{\star}]))\big\}$ such that $\nabla f(x)-\nabla
  f(x^{\star})=H(x)(x-x^{\star})$. Hence, after some computations, one
  can obtain
  \begin{equation*}
    \LL_{\rm SPD}V(x,\lambda)=-\big\langle Q(x)(x-x^{\star},\lambda-
    \lambda^{\star}),(x-x^{\star},\lambda-\lambda^{\star})\big\rangle,
  \end{equation*}
  where the symmetric matrix $Q(x)\in\rl^{(n+p)\times(n+p)}$ is given
  by
  \begin{equation*}
    Q=\begin{pmatrix}
      H+\Big(H+\dfrac{1}{\mu}A^{\top}A\Big)^{\top}\Big(H+\dfrac{1}{\mu}A^{\top}A\Big)
      + \Big(\dfrac{1}{\mu}-1\Big)
      A^{\top}A & H^{\top}A^{\top}+\dfrac{1}{\mu}A^{\top}AA^{\top}
      \\
      AH+\dfrac{1}{\mu}AA^{\top}A & AA^{\top}
    \end{pmatrix}.
  \end{equation*}
  By assumption (A1), we have that $AA^{\top}\succ0$. The Schur
  complement of $AA^{\top}$ in $Q(x)$, denoted by $Q(x)/AA^{\top}$,
  reads
  \begin{equation*}
    Q(x)/AA^{\top}=H(x)+H(x)^{\top}\big(I_{n}-A^{\top}(AA^{\top})^{-1}A\big)H(x)
    +\Big(\frac{1}{\mu}-1\Big)A^{\top}A.
  \end{equation*}
  Since $I_{n}-A^{\top}(AA^{\top})^{-1}A$ is a projection matrix,
  i.e., symmetric and idempotent, it follows that
  $I_{n}-A^{\top}(AA^{\top})^{-1}A\succeq0$. Moreover, since
  $H(x)\succ0$ (cf. Lemma~\ref{app:lm:posdef}) and $\mu\in(0,1)$, we
  conclude that $Q(x)/AA^{\top}$ is positive definite, and so is
  $Q(x)$. Thus, we have $ \LL_{\rm SPD}V(x,\lambda) \leq -
  \lambda_{\min}(Q(x))\dist((x,\lambda),\sap(L))^{2}$. More generally,
  since any $\gamma$-sublevel set $\lev V$ is compact and positively
  invariant under~\eqref{sec3:spdynamics}, and
  $\co\big\{\partial(\nabla f([x,x^{\star}]))\big\}\succ0$ for all
  $x,x^{\star}\in\rln$ (cf. Lemma~\ref{app:lm:posdef}), we conclude
  \begin{equation*}
    \LL_{\rm SPD}V(x,\lambda)\leq-\eta\dist((x,\lambda),\sap(L))^{2},
  \end{equation*}
  for all $(x,\lambda)\in\lev V$, where
  \begin{equation*}
    \eta=\min_{(x,\lambda)\in\lev V}\min_{H(x)\in\co\{\partial(\nabla
      f([x,x^\ast]))\}}\lambda_{\min}(Q(x))>0.
  \end{equation*}
  Note that both minima are attained since $\lev V$ and
  $\co\big\{\partial(\nabla f([x,x^{\star}]))\big\}$ are compact sets,
  cf.~Lemma~\ref{app:lm:compact}.

  We now proceed to quadratically upper bound the function $V$. By
  convexity of $f$ and $h$, we obtain
  \begin{align*}
    V(x,\lambda)\leq&~\langle x-x^{\star},\nabla f(x)-\nabla
    f(x^{\star})\rangle +\frac{1}{2\mu}\lVert
    h(x)\rVert^{2}+\langle\lambda-\lambda^{\star},h(x)
    \rangle\\
    &+\frac{1}{2}\dist((x,\lambda),\sap(L))^{2}\\
    =&~\frac{1}{2}\big\langle
    R(x)(x-x^{\star},\lambda-\lambda^{\star}),(x-x^{\star},
    \lambda-\lambda^{\star})\big\rangle,
  \end{align*}
  where the symmetric and positive definite matrix
  $R(x)\in\rl^{(n+p)\times(n+p)}$ is given by
  \begin{equation*}
    R(x)=\begin{pmatrix}
      2H(x)+\dfrac{1}{\mu}A^{\top}A+I_{n} & A^{\top}\\
      A                                   & I_{p}
    \end{pmatrix}.
  \end{equation*}
  Similar arguments as above yield
  %
  %
  \begin{equation*}
    V(x,\lambda)\leq\frac{1}{2}\vartheta\dist((x,\lambda),\sap(L))^{2}
  \end{equation*}
  for all $(x,\lambda)\in\lev V$, where
  \begin{equation*}
    \vartheta=\max_{(x,\lambda)\in\lev V}\max_{H(x)\in\co\{\partial(\nabla
      f([x,x^\ast]))\}}\lambda_{\max}(R(x))>0.
  \end{equation*}
  Since $\dfrac{{\rm d}}{{\rm d}t}V(x(t),\lambda(t))=\LL_{\rm
    SPD}V(x(t),\lambda(t))$ for all $t\in[0,+\infty)$, we have
  \begin{equation*}
    \frac{{\rm d}}{{\rm d}t}V(x(t),\lambda(t))\leq-\eta\dist((x(t),\lambda(t)),
    \sap(L))^{2}\leq-\frac{2\eta}{\vartheta}V(x(t),\lambda(t)),
  \end{equation*}
  for all $t\in[0,+\infty)$. Integration yields
  \begin{equation*}
    V(x(t),\lambda(t))\leq V(x_{0},\lambda_{0})\exp\left(-\frac{2\eta}{\vartheta}t\right),
  \end{equation*}
  and therefore,
  \begin{equation*}
    \dist((x,\lambda),\sap(L))\leq\sqrt{\frac{\vartheta}{\lambda_{\min}(P)}}
    \dist((x_{0},\lambda_{0}),\sap(L))\exp\left(-\frac{\eta}{\vartheta}t\right),
  \end{equation*}
  for all $t\in[0,+\infty)$. Therefore, the singleton set $\sap(L)$ is
  exponentially stable and the convergence rate of solutions
  of~\eqref{sec3:spdynamics} is upper bounded by $\eta/\vartheta$.
\end{proof}

The exponential convergence rate in Theorem~\ref{sec3:sub2:th:exp}
depends not only on the initial condition
$(x_{0},\lambda_{0})\in\rln\times\rlp$, but also on the convexity and
regularity assumptions on the objective function. However, if
$f\in\CC^{2}(\rln,\rl)$ is quadratic, then the convergence rate is
determined by $\lambda_{\min}(Q)/\lambda_{\max}(R)$, independently of
$x\in\rln$.

\begin{remark}[Local exponential stability]
  {\rm By exploiting the proof of Theorem~\ref{sec3:sub2:th:exp}, the
    exponential stability of the singleton set $\sap(L)$
    under~\eqref{sec3:spdynamics} is also guaranteed under the
    condition $\rank(A)=n$. In fact, whenever $f\in\CC^{2}(\rln,\rl)$,
    both conditions are also sufficient for the local exponential
    stability of $\sap(L)$ considering the linearization
    of~\eqref{sec3:spdynamics} around
    $(x^{\star},\lambda^{\star})\in\sap(L)$, where the Jacobian takes
    the form
    \begin{equation*}
      J_{\rm SPD}=\begin{pmatrix}
        -\nabla_{xx}L(x^{\star},\lambda^{\star}) &-\nabla_{x\lambda}L(x^{\star},
        \lambda^{\star})\\
        +\nabla_{\lambda x}L(x^{\star},\lambda^{\star}) &+\nabla_{\lambda\lambda}
        L(x^{\star},\lambda^{\star})
      \end{pmatrix}\in\rl^{(n+p)\times(n+p)}.
    \end{equation*}
    Note that the Jacobian $J_{\rm SPD}$ is negative definite if
    $\nabla_{xx}L(x^{\star},\lambda^{\star})\succ0$,
    cf.~\cite[Proposition 4.23]{DPB:82}. However, this is precisely
    the case under the hypothesis of Theorem~\ref{sec3:sub2:th:exp},
    i.e., whenever $\nabla^{2}f(x)\succ0$ for all $x\in\rln$ or
    $\rank(A)=n$.}\oprocend
\end{remark}

\begin{remark}[Connection with discrete-time algorithms]
  {\rm Compared to discrete-time solvers such as the augmented
    Lagrangian method (originally known as the method of multipliers),
    the exponential convergence rate for the continuous-time dynamics
    here corresponds to a linear convergence rate for its first-order
    Euler discretization. It is worth mentioning that we obtain the
    exponential rate under slightly weaker conditions than the ones
    usually stated in the literature for the augmented Lagrangian
    method, namely Lipschitz continuity of $\nabla f$ and strong
    convexity of $f$.}\oprocend
\end{remark}

\begin{remark}[Performance analysis under inequality and equality
  constraints] {\rm The algorithm performance bound derived considers
    convex optimization scenarios subject to equality constraints
    only. A natural question is whether the performance analysis can
    be extended to the general case including both inequality and
    equality constraints. The performance bound derived in
    Theorem~\ref{sec3:sub2:th:exp} holds true whenever the inequality
    constraints are inactive. However, we face various technical
    challenges in what concerns the analysis of the nonsmooth Lyapunov
    function~\eqref{sec3:sub1:lyapunov} whenever the
    $\ell_1$-exact-penalty terms kick in. For example, given a large
    value of the penalty parameter $\kappa$, it is challenging to
    quadratically upper bound the function $V$.}\oprocend
\end{remark}

\myclearpage
\section{Convex Optimization via Saddle-Point-Like
  Dynamics}\label{sec4}

As we noted in Section~\ref{sec3}, the condition identified in
Lemma~\ref{sec3:lm:spthm} for the penalty parameter $\kappa$ to be
exact relies on knowledge of the dual solution set.  In turn,
exactness is required to ensure that the saddle-point
dynamics~\eqref{sec3:spdynamics} converges to a solution
of~\eqref{sec3:probstatement}. Motivated by these observations and
building on our results above, in this section we propose
discontinuous saddle-point-like dynamics that do not rely on a priori
knowledge of the penalty parameter $\kappa$ and converge to a solution
of the optimization problem.

Let $G=\{x\in\rln\mid g(x)\leq0_{m}\}$ denote the inequality
constraint set associated with the convex
program~\eqref{sec3:probstatement}. Let the set-valued flow
$F:G\times\rlp\rightrightarrows\rln$ be defined by
\begin{equation*}
  F(x,\lambda)=-\nabla\Big(\frac{1}{2\mu}\lVert h(x)\rVert^{2}\Big)-\nabla_{x}
  \langle\lambda,h(x)\rangle-\partial f(x).
\end{equation*}
The choice is motivated by the fact that, for $(x,\lambda)\in\inter
G\times\rlp$, we have
$-\partial_{x}L(x,\lambda)=F(x,\lambda)$. Consider now the
saddle-point-like dynamics defined over $G\times\rlp$,
\begin{equation}
  \renewcommand{\theequation}{SPLD}\tag{\theequation}\label{sec4:spldynamics}
  \begin{cases}
    (\dot{x},\dot{\lambda}) \in (P_{T_G}(F),\partial_{\lambda}L)
    (x,\lambda) ,
    \\
    (x(0),\lambda(0))=(x_{0},\lambda_{0})\in G\times\rlp,
  \end{cases}
\end{equation}
where the projection operator $P_{T_G}$ is defined
in~\eqref{sec2:sub2:projop}. Since the mapping $(x,\lambda)\mapsto
F(x,\lambda)$ is locally bounded, upper semi-continuous and takes
nonempty, convex and compact values, the existence of local solutions
$(x,\lambda)\in\AC([0,t_{+}),G\times\rlp)$ of~\eqref{sec4:spldynamics}
is guaranteed by Lemma~\ref{sec2:sub2:lm:prexistence}.

Our strategy to show that the saddle-point-like
dynamics~\eqref{sec4:spldynamics} also converge to the set of saddle
points is to establish that, in fact, its solutions are also solutions
of the saddle-point dynamics~\eqref{sec3:spdynamics} when the penalty
parameter $\kappa$ is sufficiently large. To make this precise, we
first investigate the explicit computation of the projection
operator~$P_{T_G}$. Recall that $T_{G}(x)$ and $N_{G}(x)$ denote the
tangent and normal cone of $G\subset\rln$ at $x\in G$,
respectively. Let the set of unit outward normals to $G$ at $x\in\bd
G$ be defined by
\begin{equation*}
  N_{G}^{\sharp}(x)=N_{G}(x)\cap\bd\overline{\BB}(0,1).
\end{equation*}
The following geometric interpretation of $P_{T_G}$ is well-known in
the literature of locally projected dynamical
systems~\cite{AN-DZ:96,MP-MP:02}:
\begin{enumerate}[(i)]
\item if $(x,\lambda)\in\inter G\times\rlp$, then
  $P_{T_G(x)}(F(x,\lambda))=F(x,\lambda)$;
	\item if $(x,\lambda)\in\bd G\times\rlp$, then
		\begin{equation*}
			P_{T_G(x)}(F(x,\lambda))=\bigcup_{\xi\in F(x,\lambda)}\xi-\max\big\{0,
			\langle\xi,n^{\star}(x,\xi)\rangle\big\}n^{\star}(x,\xi),
		\end{equation*}
	where
	\begin{equation}\label{sec4:subproj}
		n^{\star}(x,\xi)\in\argmax_{n\in N_{G}^{\sharp}(x)}\langle\xi,n\rangle.
	\end{equation}
\end{enumerate}
Note that if $\{\xi\}\cap T_{G}(x)\neq\emptyset$ for some
$(x,\lambda)\in\bd G\times\rlp$ and $\xi\in F(x,\lambda)$, then
$\sup_{n\in N_{G}^{\sharp}(x)}\langle\xi,n\rangle\leq0$, and by
definition of $P_{T_G}$, no projection needs to be performed. The
following result establishes the existence and uniqueness of the
maximizer $n^{\star}(x,\xi)$ of~\eqref{sec4:subproj}~whenever
$\{\xi\}\cap T_{G}(x)=\emptyset$.

\begin{lemma}[Existence and uniqueness]\label{sec4:lm:exanduni}
  Let $(x,\lambda)\in\bd G\times\rlp$. If there exists $\xi\in
  F(x,\lambda)$ such that $\sup_{n\in
    N_{G}^{\sharp}(x)}\langle\xi,n\rangle>0$, then the maximizer
  $n^{\star}(x,\xi)$ of~\eqref{sec4:subproj} exists and is unique.
\end{lemma}
\begin{proof}
  Let $(x,\lambda)\in\bd G\times \rlp$ and suppose there exists
  $\xi\in F(x,\lambda)$ such that $\sup_{n\in
    N_{G}^{\sharp}(x)}\langle\xi,n\rangle>0$. By definition, the
  normal cone $N_{G}(x)$ of $G$ at $x\in\bd G$ is closed and
  convex. Existence of $n^{\star}(x,\xi)$ follows from compactness of
  the set $N_{G}^{\sharp}(x)$. Now, let $\tilde{n}^{\star}(x,\xi)$ and
  $\hat{n}^{\star}(x,\xi)$ be two distinct maximizer
  of~\eqref{sec4:subproj} such that
  $\langle\xi,\tilde{n}^{\star}(x,\xi)\rangle>0$ and
  $\langle\xi,\hat{n}^{\star}(x,\xi)\rangle>0$. Convexity implies
  $(\tilde{n}^{\star}(x,\xi) +
  \hat{n}^{\star}(x,\xi))/\lVert\tilde{n}^{\star}(x,\xi)+\hat{n}^{\star}(x,\xi)\rVert\in
  N_{G}^{\sharp}(x)$. Therefore, it follows
  \begin{align*}
    \frac{\langle\xi,\tilde{n}^{\star}(x,\xi)+\hat{n}^{\star}(x,\xi)\rangle}
    {\lVert\tilde{n}^{\star}(x,\xi)+\hat{n}^{\star}(x,\xi)\rVert}=
    \frac{2\langle\xi,\tilde{n}^{\star}(x,\xi)\rangle}{\lVert\tilde{n}^{\star}
      (x,\xi)+\hat{n}^{\star}(x,\xi)\rVert}>\langle\xi,\tilde{n}^{\star}(x,\xi)
    \rangle,
  \end{align*}
  which contradicts the fact that $\tilde{n}^{\star}(x,\xi)$
  maximizes~\eqref{sec4:subproj}.
\end{proof}

We note that the computational complexity of
solving~\eqref{sec4:subproj}~depends not only on the problem
dimensions $n,p,m>0$, but also on the convexity and regularity
assumptions of the problem data, i.e., on $f$, $h$ and $g$. The
following result establishes a relationship between the solutions
of~\eqref{sec4:spldynamics} and~\eqref{sec3:spdynamics}.

\begin{proposition}[Relationship of solutions]\label{sec4:pr:relsol}
  Let $(x,\lambda):[0,+\infty)\to G\times\rlp$ be a solution
  of~\eqref{sec4:spldynamics} starting from $(x_{0},\lambda_{0})\in
  G\times\rlp$. Then, there exists $\kappa>0$ such that the solution
  is also a solution of~\eqref{sec3:spdynamics}.
\end{proposition}
\begin{proof}
  Since~\eqref{sec3:spdynamics} and~\eqref{sec4:spldynamics} are
  identical on $\inter G\times\rlp$, it suffices to focus our
  attention on showing
  $P_{T_G(x)}(F(x,\lambda))\subset-\partial_{x}L(x,\lambda)$ for
  points $(x,\lambda)\in\bd G\times\rlp$. Take $\xi\in F(x,\lambda)$
  and suppose $\{\xi\}\cap T_{G}(x)\neq\emptyset$. In this case, it
  follows $\sup_{n\in N_{G}^{\sharp}(x)}\langle\xi,n\rangle\leq0$, and
  by definition of $P_{T_G}$, we have $P_{T_G}(\{\xi\})=\xi$. Clearly,
  \begin{equation*}
    P_{T_G}(\{\xi\})\in\xi-\kappa~\sum_{\mathclap{k\in
        K^{\circ}(x)}}~\co\{\{0\},\partial g_{k}(x)\}, 
  \end{equation*}
  for any $\kappa>0$, where $K^{\circ}(x)=\{k\in\until{m}\mid
  g_{k}(x)=0\}$. Suppose now $\{\xi\}\cap T_{G}(x)=\emptyset$, i.e.,
  $\sup_{n\in N_{G}^{\sharp}(x)}\langle\xi,n\rangle>0$. By
  Lemma~\ref{sec4:lm:exanduni}, there exists a unique
  $n^{\star}(x,\xi)\in N_{G}^{\sharp}(x)$
  maximizing~\eqref{sec4:subproj} such that $P_{T_G}(\{\xi\}) =
  \xi-\max\big\{0,\langle\xi,n^{\star}(x,\xi)\rangle\big\}n^{\star}(x,\xi)$. Now,
  by rescaling $n^{\star}(x,\xi)$ by some constant
  $\sigma^{\star}(x,n^{\star})>0$ minimizing
  \begin{equation*}
    \min\Big\{\sigma~\Big\vert~\frac{1}{\sigma}n^{\star}(x,\xi)\in
    \sum_{\mathclap{k\in K^{\circ}(x)}}~\co\{\{0\},\partial g_{k}(x)\}\Big\},
  \end{equation*}
  the choice
  $\kappa\geq\sigma^{\star}(x,n^{\star})\max\big\{0,\langle\xi,n^{\star}(x,\xi)\rangle\big\}$
  guarantees that
  \begin{equation*}
    P_{T_G}(\{\xi\})\in\xi-\kappa~\sum_{\mathclap{k\in
        K^{\circ}(x)}}~\co\{\{0\},\partial g_{k}(x)\}.
  \end{equation*}
  Now, since $\xi\in F(x,\lambda)$ is arbitrary, if
  \begin{equation*}
    \kappa\geq\max_{\xi\in F(x,\lambda)}\sigma^{\star}(x,n^{\star})\max
    \big\{0,\langle\xi,n^{\star}(x,\xi)\rangle\big\},
  \end{equation*}
  where $F(x,\lambda)$ is nonempty, convex and compact, we conclude
  that $P_{T_G(x)}(F(x,\lambda))\subset-\partial_{x}L(x,\lambda)$. More
  generally, let $V$ be defined as in Remark~\ref{sec3:rm:lasalle} for
  some primal-dual solution $(x^{\star},\lambda^{\star})$
  of~\eqref{sec3:probstatement}. Then, any choice
  \begin{equation*}
    \kappa\geq\max_{(x,\lambda)\in\lev V\cap(G\times\rlp)}\max_{\xi\in
      F(x,\lambda)}\sigma^{\star}(x,n^{\star})\max\{0,\langle\xi,n^{\star}(x,\xi)\rangle\}, 
  \end{equation*}
  guarantees that any solution of~\eqref{sec4:spldynamics} starting in
  $\lev V\cap(G\times\rlp)$ is also a solution
  of~\eqref{sec3:spdynamics}, concluding the proof.
\end{proof}

The arbitrariness of the choice of $\gamma$ in
Proposition~\ref{sec4:pr:relsol} ensures that, given any solution
of~\eqref{sec4:spldynamics}, there exists $\kappa$ such that the
solution is also a solution of~\eqref{sec3:spdynamics}. Note that, in
general, the set of solutions of~\eqref{sec3:spdynamics} is richer
than the set of solutions
of~\eqref{sec4:spldynamics}. Figure~\ref{sec4:fig1} illustrates the
effect of an increasing penalty parameter $\kappa$
on~\eqref{sec3:spdynamics}. The combination of
Theorem~\ref{sec3:sub1:th:lyapunov} and
Proposition~\ref{sec4:pr:relsol} leads immediately to the following
result.

\begin{corollary}[Asymptotic convergence] 
  Any solution $(x,\lambda)\in\AC([0,+\infty),G\times\rlp)$
  of~\eqref{sec4:spldynamics} starting from a point in $G\times\rlp$
  converges asymptotically to a point in the set $S\times\rlp$, where
  $S\subset\rln$ is the set of solutions
  of~\eqref{sec3:probstatement}.
\end{corollary}

The fact that the saddle-point-like dynamics~\eqref{sec4:spldynamics}
do not incorporate any knowledge of the penalty parameter $\kappa$
makes them amenable to distributed implementation in multi-agent
systems. This is the point we address in the next section.

\begin{figure}[!t]
	\centering
	\subfloat[$\kappa\ngeq\max\limits_{\xi\in F(x,\lambda)}
	\sigma^{\star}(x,n^{\star})\max\big\{0,\langle\xi,n^{\star}(x,\xi)\rangle
	\big\}$.]{\makebox[6.2cm][c]
	{\small
	\psfrag{G}{$G$}
	\psfrag{x}{$x$}
	\psfrag{k}{$\kappa$}
	\psfrag{N}{$N_{G}(x)$}
	\psfrag{F}{$F(x,\lambda)$}
	\psfrag{P}{$P_{T_G(x)}(F(x,\lambda))$}
		{\includegraphics[width=0.28\textwidth]{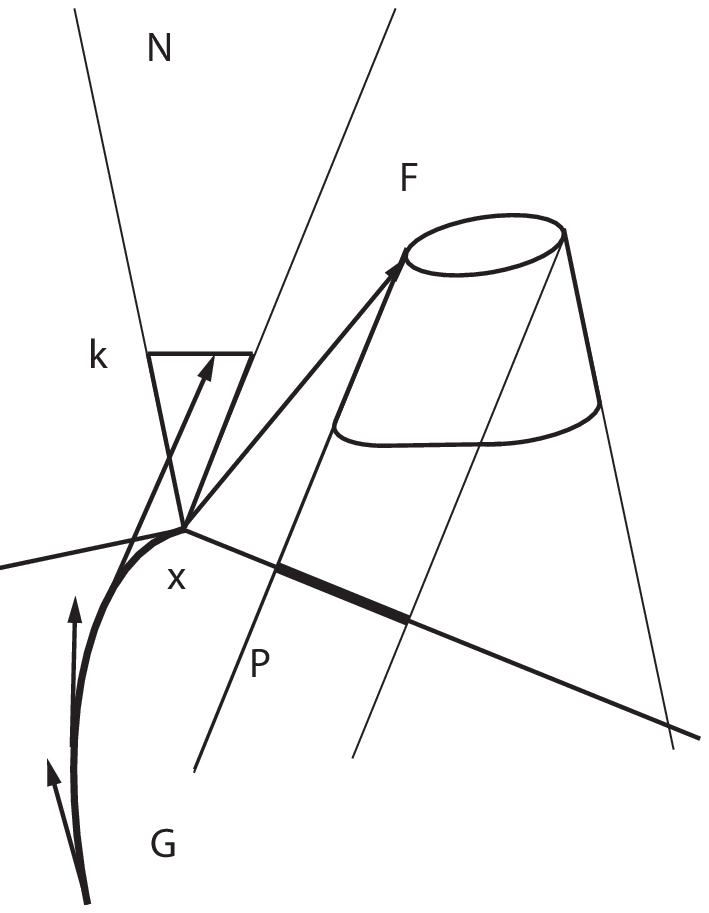}}
		}} ~~~
	\subfloat[$\kappa\geq\max\limits_{\xi\in F(x,\lambda)}
	\sigma^{\star}(x,n^{\star})\max\big\{0,\langle\xi,n^{\star}(x,\xi)\rangle
	\big\}$.]{\makebox[6.2cm][c]
	{\small
	\psfrag{G}{$G$}
	\psfrag{x}{$x$}
	\psfrag{k}{$\kappa$}
	\psfrag{N}{$N_{G}(x)$}
	\psfrag{F}{$F(x,\lambda)$}
	\psfrag{P}{$P_{T_G(x)}(F(x,\lambda))$}
	{\includegraphics[width=0.28\textwidth]{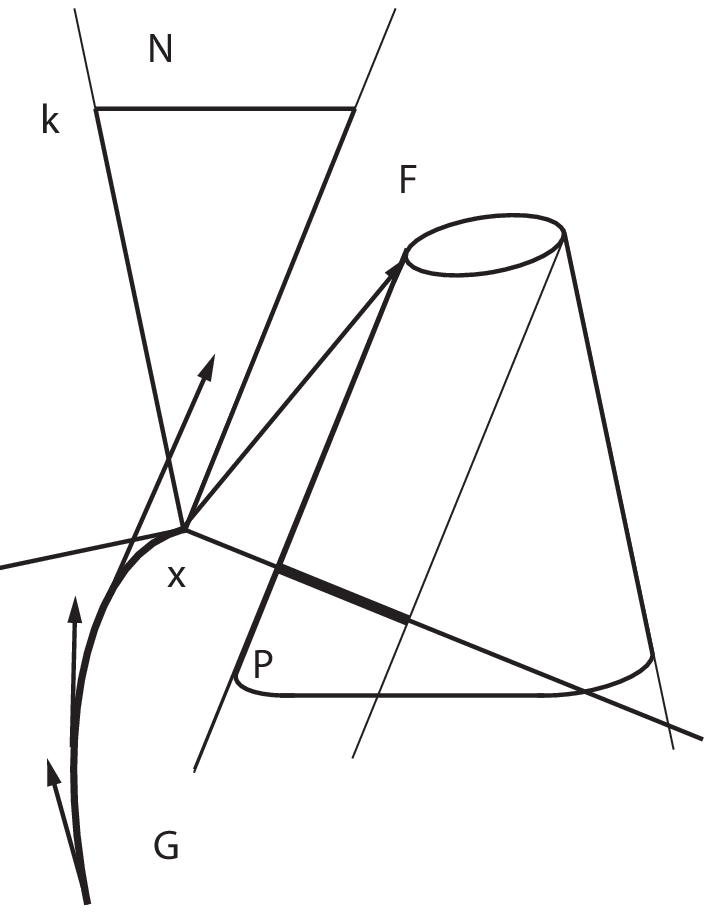}}
	}}
	\caption{Illustration of the effect of an increasing penalty
          parameter $\kappa>0$ on~\eqref{sec3:spdynamics}. For a fixed
          $(x,\lambda)\in \bd G\times\rlp$, the set
          $-\partial_{x}L(x,\lambda)=F(x,\lambda)-\kappa\sum_{k=1}^{m}\co\{\{0\},\partial
          g_{k}(x)\}$ enlarges as the penalty parameter $\kappa>0$
          increases (a), until the exact-penalty value is
          reached/exceeded (b) where the inclusion
          $P_{T_G(x)}(F(x,\lambda))\subset-\partial_{x}L(x,\lambda)$
          holds.}\label{sec4:fig1} 
        \vspace{-10pt}
\end{figure}


\myclearpage
\section{Distributed Implementation}\label{sec5}

In this section, we describe the requirements that ensure that the
proposed saddle-point-like algorithm~\eqref{sec4:spldynamics} is
well-suited for distributed implementation.

Consider a network of $n\in\NN$ agents whose communication topology is
represented by an undirected and connected graph $\GG=(\VV,\EE)$,
where $\VV=\until{n}$ is the vertex set and $\EE\subset\VV\times\VV$
is the (symmetric) edge set. The objective of the agents is to
cooperatively solve the constraint minimization
problem~\eqref{sec3:probstatement}. We assume that the aggregate
objective function $f$ is additively separable, i.e.,
$f(x)=\sum_{i=1}^{n}f_{i}(x_{i})$, where $f_{i}$ and $x_{i}\in\rl$
denote the local objective function and state associated with agent
$i\in\until{n}$, respectively. Additionally, we assume that the
constraints of~\eqref{sec3:probstatement} are compatible with the
network topology described by $\GG$. Formally, we say the inequality
constraints $g_{k}(x)\leq0$, $k\in\until{m}$, are compatible with
$\GG$ if $g_{k}$ can be expressed as a function of some components of
the network state $x=(x_{1},\dots,x_{n})\in\rln$, which induce a
complete subgraph of $\GG$. A similar definition can be stated for the
equality constraints $h_{\ell}(x)=0$, $\ell\in\until{p}$.

In this network scenario, if $(x,\lambda)\in\inter G\times\rlp$, then
each agent $i\in\until{n}$ implements its primal
dynamics~\eqref{sec4:spldynamics}, where
$P_{T_G(x)}(F(x,\lambda))=F(x,\lambda)$, i.e.,
\begin{equation*}
  \dot{x}_{i}+~\sum_{\mathclap{\{\ell:a_{\ell i}\neq0\}}}~a_{\ell i}
  \bigg(\frac{1}{\mu}\bigg(~~~\sum_{\mathclap{\{j:a_{\ell j}\neq0\}}}~
  a_{\ell j} x_{j}-b_{\ell}\bigg)+\lambda_{\ell}\bigg)\in-\partial f_i(x_i),
\end{equation*}
and some dual dynamics~\eqref{sec4:spldynamics}, i.e.,
\begin{equation*}
  \dot{\lambda}_{\ell}=~\sum_{\mathclap{\{i:a_{\ell i}\neq0\}}}~a_{\ell i}x_{i}
  -b_{\ell},
\end{equation*}
where $\ell\in\until{p}$, corresponding to the Lagrange multipliers
for the constraints that the agent is involved in (alternatively)

Hence, in order for agent $i$ to be able to
implement its corresponding primal dynamics, it also needs access to
certain dual components $\lambda_{\ell}$ for which $a_{\ell
  i}\neq0$. If $(x,\lambda)\in\bd G\times\rlp$, then each agent
$i\in\until{n}$ implements the locally projected
dynamics~\eqref{sec4:spldynamics}, i.e.,
\begin{equation*}
  \dot{x}_{i}\in\bigcup_{\xi_{i}\in F_{i}(x,\lambda)}\xi_{i}-\max\bigg\{0,
  ~\sum_{\mathclap{\{j:n_{j}^{\star}\neq0\}}}~\xi_{j}n_{j}^{\star}(x,\xi)
  \bigg\}n_{i}^{\star}(x,\xi),
\end{equation*}
and the dual dynamics~\eqref{sec4:spldynamics} described above. Hence,
if the states of some agents are, at some time instance, involved in
the active inequality constraints, the respective agents need to
solve~\eqref{sec4:subproj}. We say that the saddle-point-like
algorithm~\eqref{sec4:spldynamics} is distributed over $\GG=(\VV,\EE)$
when the following conditions are satisfied:
\begin{enumerate}[\hspace{10pt}({C}1)]
\item The network constraints $h$ and $g$ are compatible with the graph $\GG$;
\item Agent $i$ knows its state $x_{i}\in\rl$ and its objective
  function $f_{i}$;
\item Agent $i$ knows its neighbors' states $x_{j}\in\rl$, their
  objective functions $f_{j}$, and
  \begin{enumerate}[(i)]
  \item the non-zero elements of every row of $A\in\rlpn$, and every
    $b_{\ell}\in\rl$ for which $a_{\ell i}\neq0$, and
  \item the active inequality constraints $g_{k}$ in which agent $i$
    and its neighbors are involved.
  \end{enumerate}
\end{enumerate}
(Note that, under these assumptions, it is also possible to have, for
each Lagrange multiplier, only one agent implement the corresponding
dual dynamics and then share the computed value with its neighboring
agents -- which are the ones involved in the corresponding equality
constraint).

Note that the saddle-point-like algorithm~\eqref{sec4:spldynamics} can
solve optimization scenarios where the agents' states belong to an
arbitrary Euclidean space. In contrast to consensus-based distributed
algorithms where each agent maintains, communicates and updates an
estimate of the complete solution vector of the optimization problem,
the saddle-point-like algorithm~\eqref{sec4:spldynamics} only requires
each agent to store and communicate its own component of the solution
vector. Thus, the algorithm scales well with respect to the number of
agents in the network. The following examples illustrate an
application of the above results to nonsmooth convex optimization
scenarios over a network of agents.

\begin{example}[Saddle-point-like dynamics for nonsmooth convex
  optimization]
  {\rm Consider a network of $n=50$ agents that seek to cooperatively
    solve the nonsmooth convex optimization problem
    \begin{equation}\label{sec5:ex1:problem}
      \begin{aligned}
        & \underset{x\in\rln}{\text{minimize}}
        & &\sum_{\mathclap{i\in\until{n}}}~x_{i}^{4}/4+\lvert x_{i}\rvert\\
        & \text{subject to}
        & &\circu\nolimits_{n}(0,1,1/2)x=\on/5,\\
        &&&~x_{i}\leq1/2,\ i\in\until{n},
      \end{aligned}
    \end{equation}
    where $x_{i}\in\rl$ denotes the state associated with agent
    $i\in\until{n}$, and $\circu\nolimits_{n}(0,1,1/2)$ is the
    tridiagonal circulant matrix~\cite{FB-JC-SM:08cor} encoding the
    network topology in its sparsity structure.  Although this
    specific example is academic, examples belonging to the same class
    of optimization problems arise in a variety of networked
    scenarios, see e.g.~\cite{PW-MDL:09}.  The generalized gradient of
    $f_{i}$ at $x_{i}\in\rl$ is
    \begin{equation*}
      \partial f_{i}(x_{i})=\begin{cases}
        \{x_{i}^{3}+1\}, &\text{if $x_{i}>0$},\\
        [-1,1], &\text{if $x_{i}=0$},\\
        \{x_{i}^{3}-1\}, &\text{if $x_{i}<0$}.
      \end{cases}
    \end{equation*}
    Figure~\ref{sec5:fig1} illustrates the asymptotic convergence of
    solutions of the saddle-point-like
    dynamics~\eqref{sec4:spldynamics} to the set
    $\sap(L)=\{x^{\star}\}\times M$, where
    $M=\{\lambda^{\star}\in\rlp\mid\lambda^{\star}\in-(\circu\nolimits_{n}(0,1,1/2)\circu\nolimits_{n}(0,1,1/2)^{\top})^{-1}\circu\nolimits_{n}(0,1,1/2)\partial
    f(x^{\star})\}$.}\oprocend
\end{example}
\begin{figure*}[!t]
  \centering\subfloat[Network state evolution]
  {\includegraphics[width=0.33\linewidth]{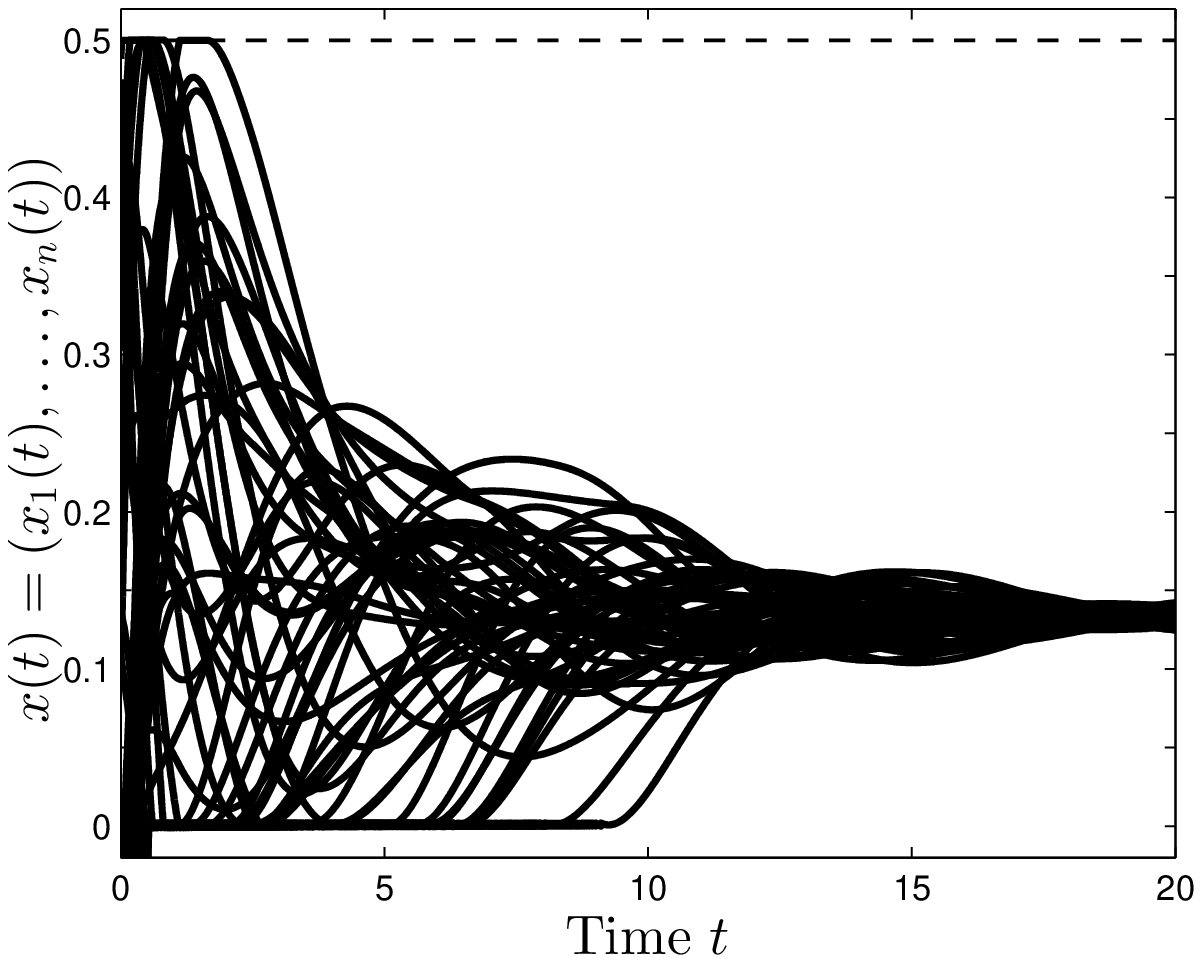}}
  \subfloat[Multiplier evolution]
  {\includegraphics[width=0.33\linewidth]{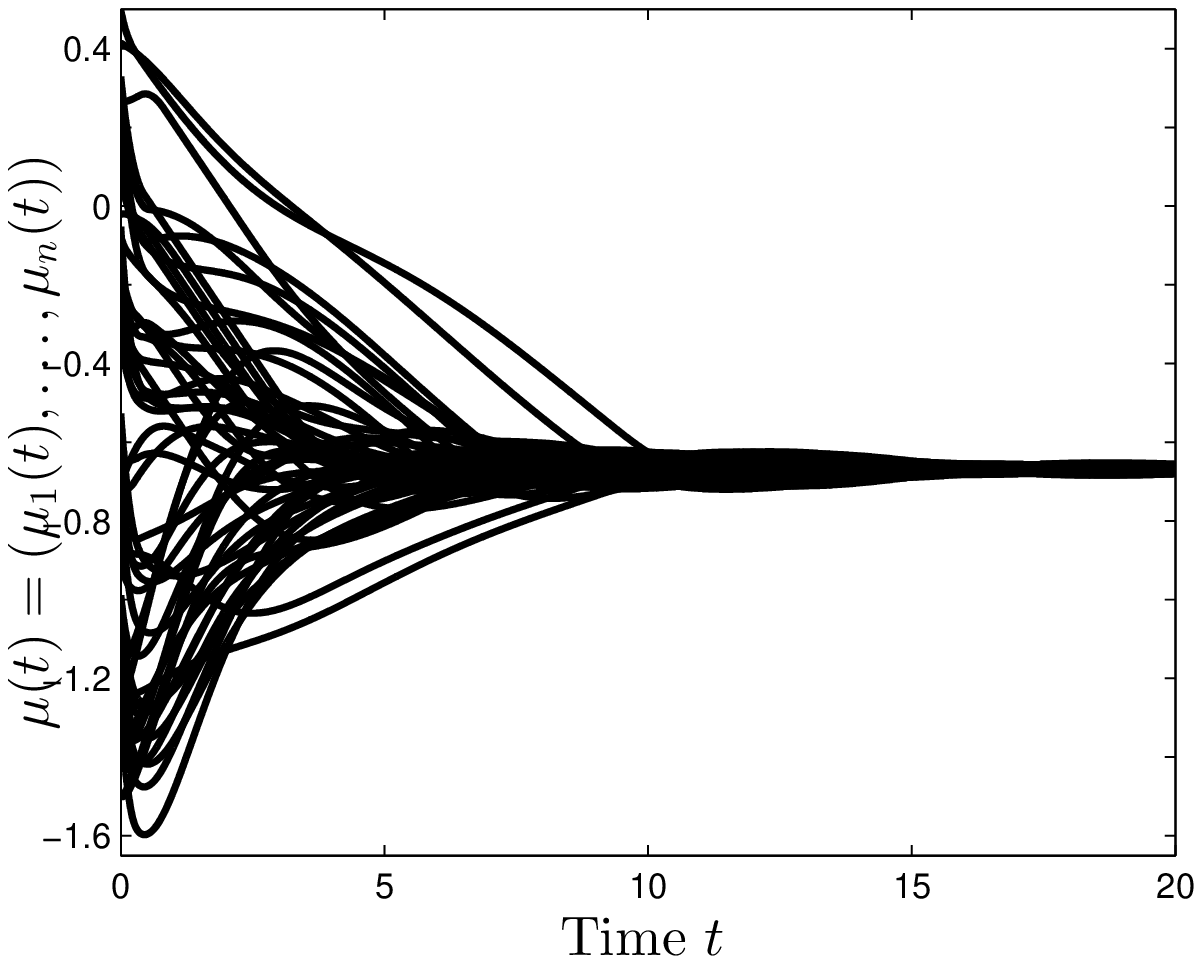}}
  \subfloat[Convergence]
  {\includegraphics[width=0.33\linewidth]{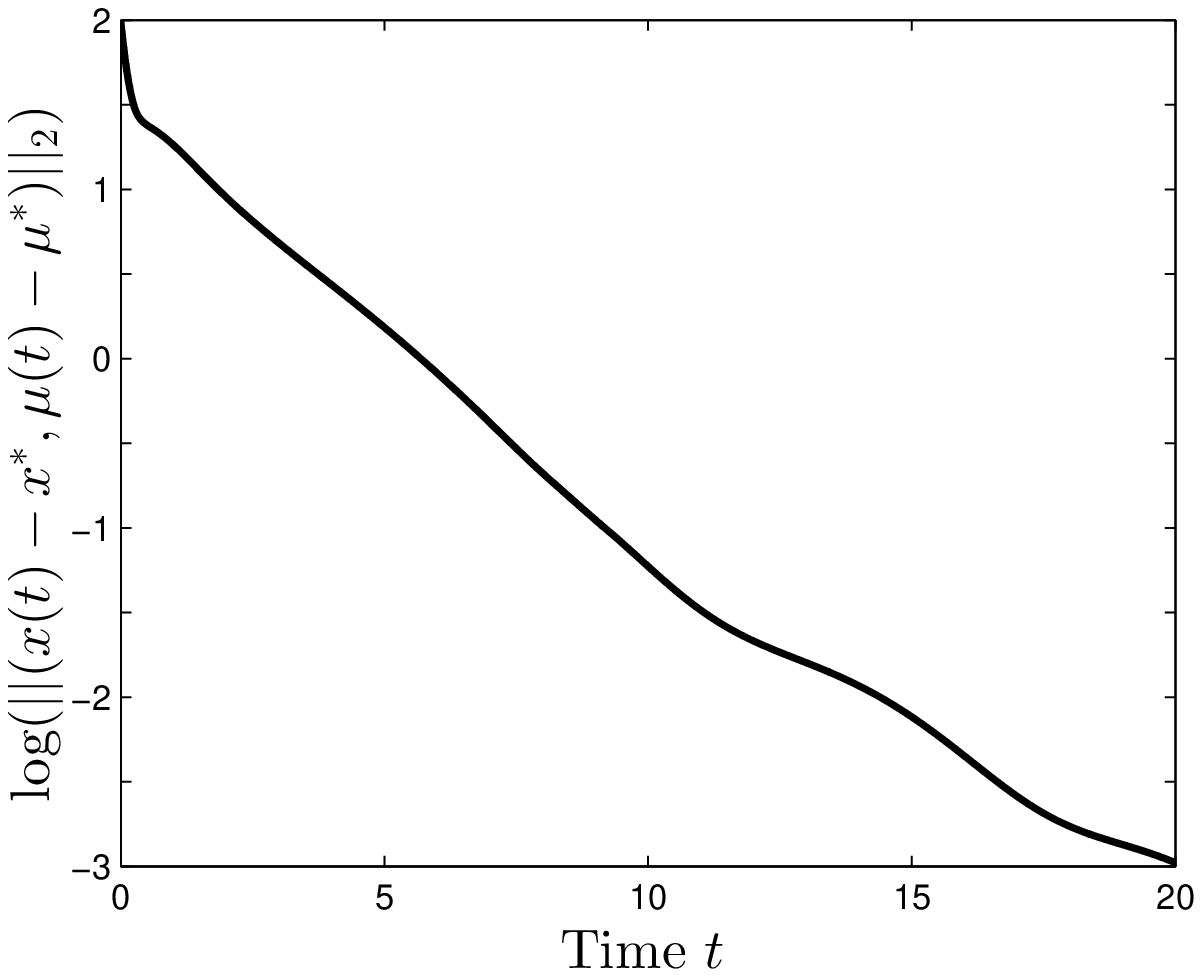}}
  \caption{(a) The network state evolution of
    algorithm~\eqref{sec4:spldynamics} solving the nonsmooth convex
    program~\eqref{sec5:ex1:problem}. The projection operator prevents
    the solutions of~\eqref{sec4:spldynamics} from violating the
    inequality constraints (depicted with a dashed line)
    of~\eqref{sec5:ex1:problem}. (b) The Lagrange multiplier evolution
    associated with the equality constraints
    of~\eqref{sec5:ex1:problem}. The initial conditions are randomly
    chosen from the interval $[-3/2,1/2]$. (c) Since the aggregate
    objective function of~\eqref{sec5:ex1:problem} is strictly convex,
    the solutions of~\eqref{sec4:spldynamics} converge asymptotically
    to the set $\sap(L)$.}\label{sec5:fig1}
	\vspace{-10pt}
\end{figure*}

\begin{example}[Saddle-point dynamics for equality constrained
  optimization]
  {\rm Consider a network of $n=50$ agents whose objective is to
    cooperatively solve the convex optimization problem
    \begin{equation}\label{sec5:ex2:problem}
      \begin{aligned}
        & \underset{x\in\rln}{\text{minimize}}
        & &\sum_{\mathclap{i\in\until{n}}}~f_{i}(x_{i})\\
        & \text{subject to}
        & &\trid\nolimits_n(1/2,1,-1/10)x=\on,\\
      \end{aligned}
    \end{equation}
    where $f_{i}\in\CCC(\rl,\rl)$ defined by
    \begin{equation*}
      f_{i}(x_{i})=\begin{cases}
        x_{i}^{2}, &\text{if $x_{i}\geq0$},\\
        x_{i}^{2}/2, &\text{if $x_{i}<0$},
      \end{cases}
    \end{equation*}
    is the objective function associated with agent
    $i\in\until{n}$, and $\trid\nolimits_{n}(1/2,1,-1/10)$ is the
    tridiagonal Toeplitz matrix~\cite{FB-JC-SM:08cor} of dimension
    $n\times n$. The network topology is encoded in the sparsity
    structure of $\trid\nolimits_{n}(1/2,1,-1/10)$.  The gradient
    of $f_{i}$ at $x_{i}\in\rl$ is $\nabla
    f_{i}(x_{i})=\max\{x_{i},2x_{i}\}$, and the generalized
    Hessian of $f_{i}$ at $x_{i}\in\rl$ is
    \begin{equation*}
      \partial(\nabla f_{i})(x_{i})=\begin{cases}
        \{2\}, &\text{if $x_{i}>0$},\\
        [1,2], &\text{if $x_{i}=0$},\\
        \{1\}, &\text{if $x_{i}<0$}.
      \end{cases}
    \end{equation*}
    Figure~\ref{sec5:fig2} illustrates the convergence and performance
    of solutions of the saddle-point dynamics~\eqref{sec3:spdynamics}
    to the set of saddle-points
    $\sap(L)=\{x^{\star}\}\times\{\lambda^{\star}\}$, where
    $\lambda^{\star}=-(\trid\nolimits_{n}(1/2,1,-1/10)\trid\nolimits_{n}(1/2,1,-1/10)^{\top})^{-1}\trid\nolimits_{n}(1/2,1,-1/10)\nabla
    f(x^{\star})$.}\oprocend
\end{example}

\begin{figure*}[!t] \centering \subfloat[Network state evolution]
  {\includegraphics[width=0.33\linewidth]{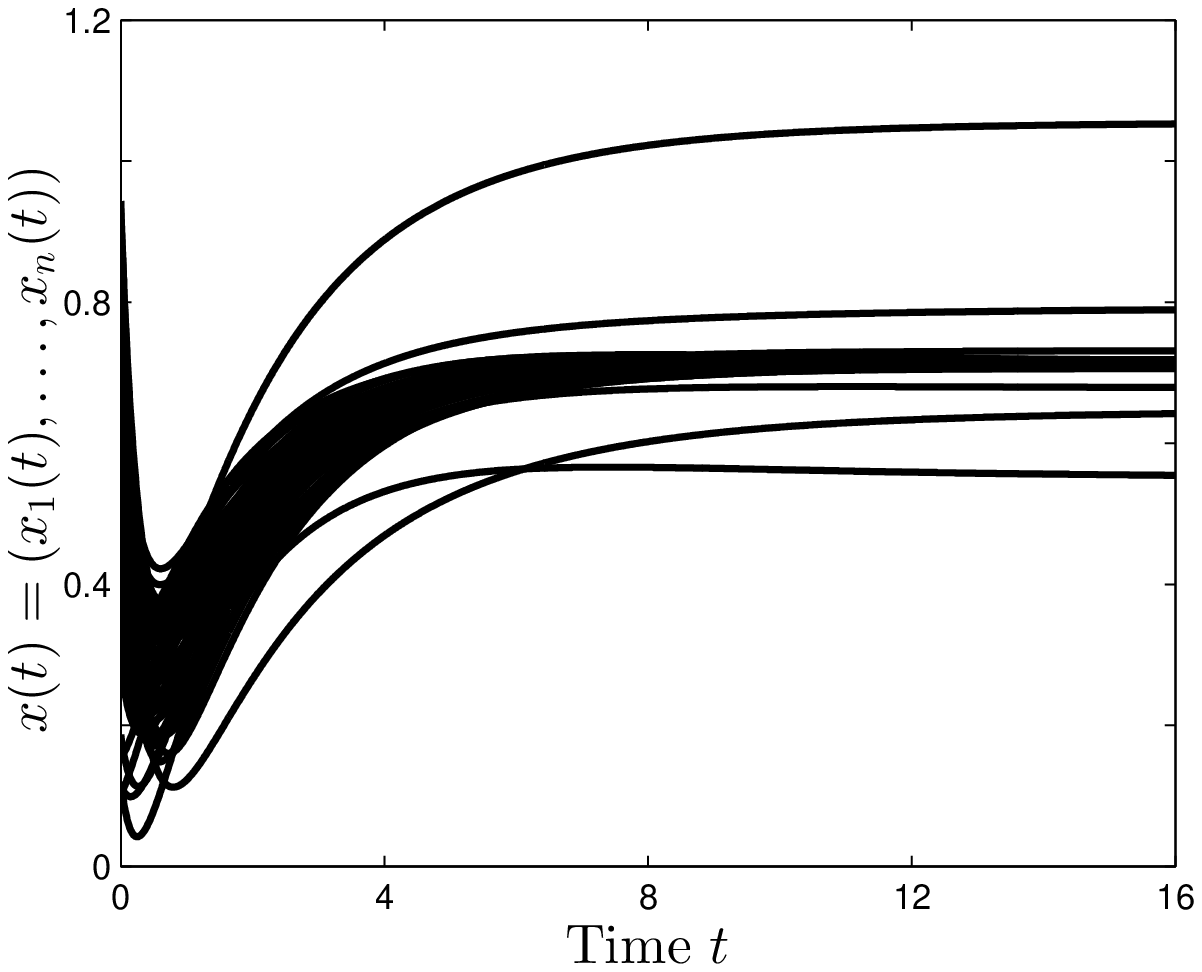}}
  \subfloat[Multiplier evolution]
  {\includegraphics[width=0.33\linewidth]{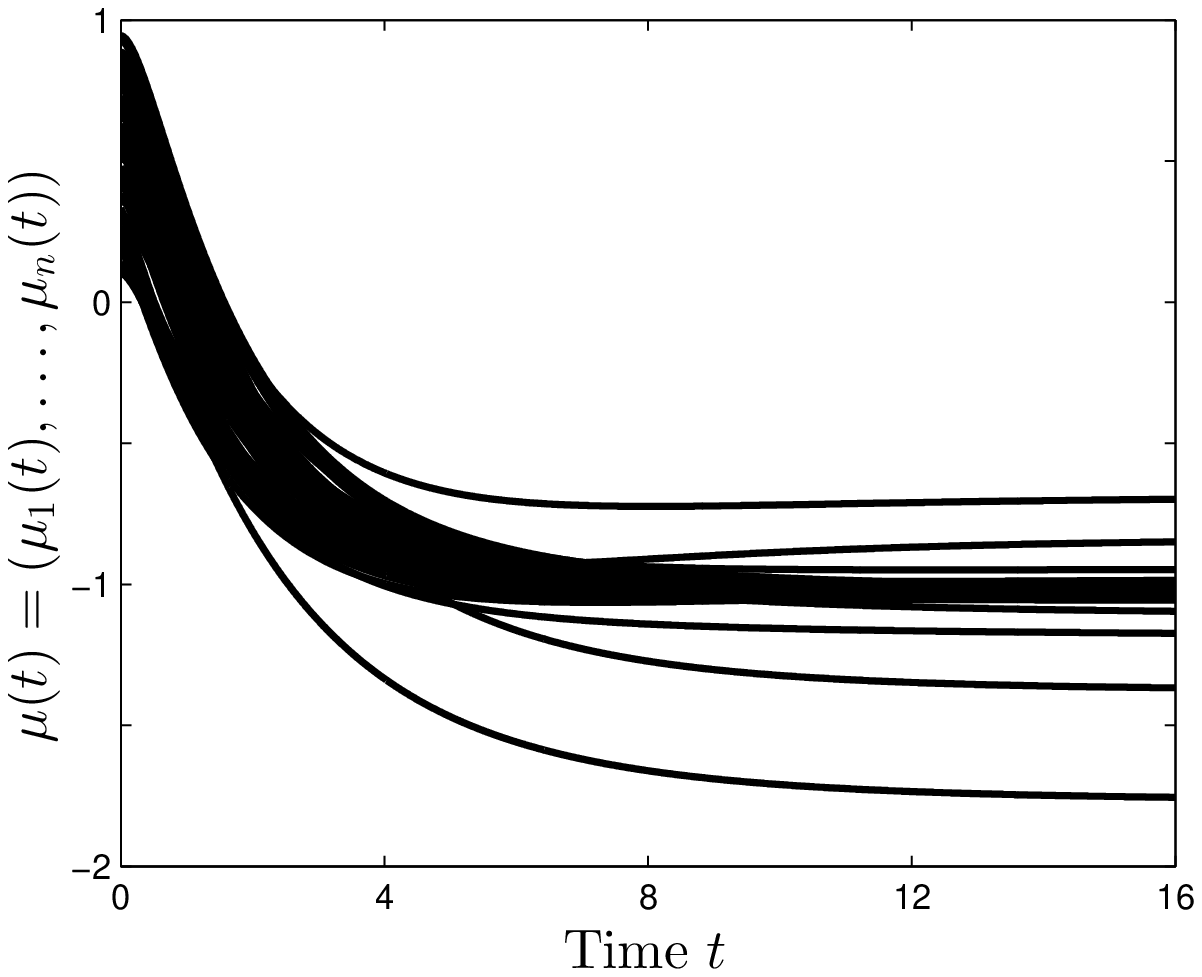}}
  \subfloat[Performance bound]
  {\includegraphics[width=0.33\linewidth]{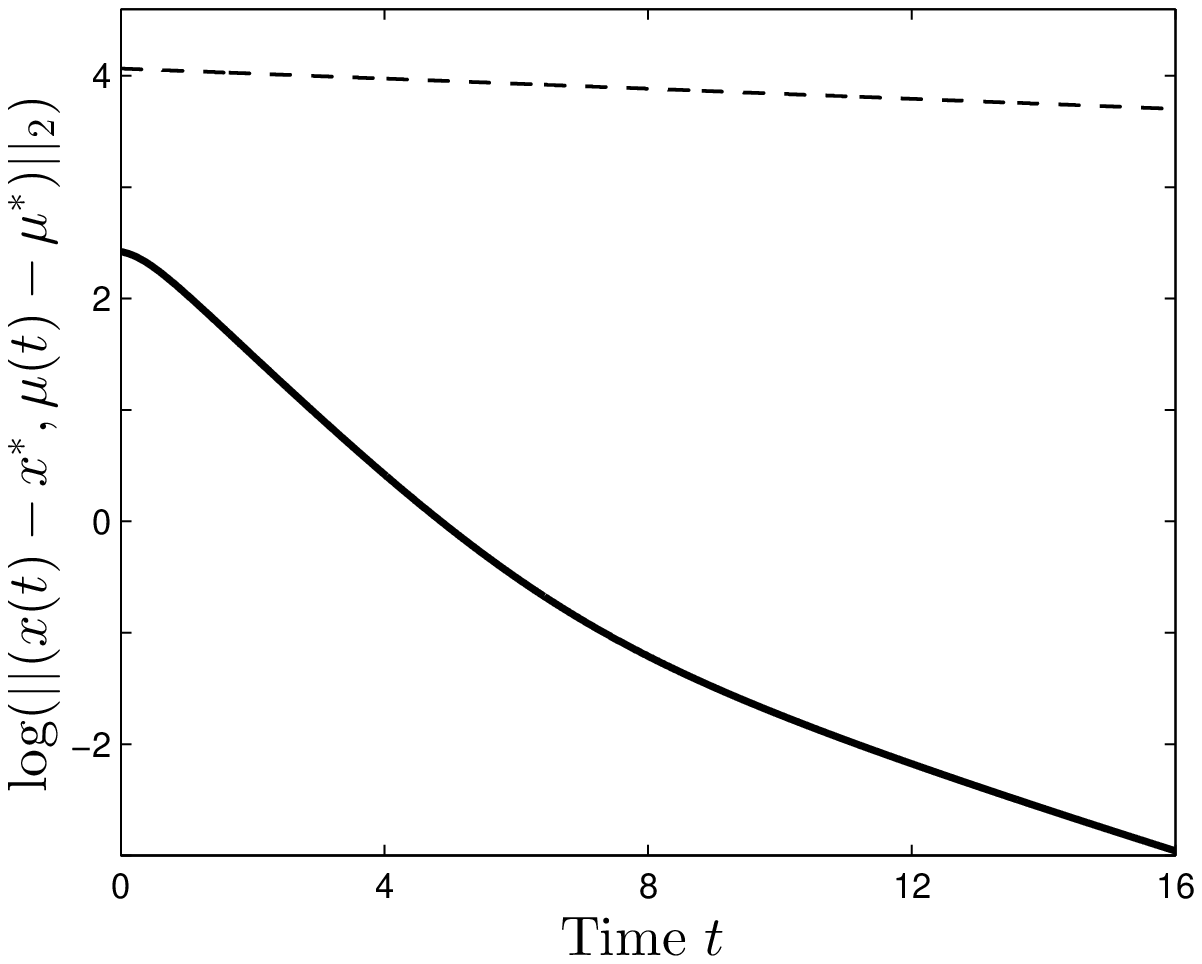}}
	\caption{(a) The network state evolution of
          algorithm~\eqref{sec3:spdynamics} solving the convex
          program~\eqref{sec5:ex2:problem}. (b) The Lagrange
          multiplier evolution associated with the equality
          constraints of~\eqref{sec5:ex2:problem}. The initial
          conditions are randomly chosen in the interval $(0,1]$. (c)
          Since $\partial(\nabla f)\succ0$ (in fact, the aggregate
          objective function of~\eqref{sec5:ex2:problem} is strongly
          convex), the hypothesis of Theorem~\ref{sec3:sub2:th:exp} is
          satisfied and thus, the solutions of~\eqref{sec3:spdynamics}
          converge to the singleton set
          $\sap(L)=\{x^{\star}\}\times\{\lambda^{\star}\}$ within the
          exponential performance bound (depicted by the dashed
          line).}\label{sec5:fig2}
	\vspace{-10pt}
\end{figure*}

\myclearpage
\section{Conclusions}\label{sec6}

We have investigated the design of continuous-time solvers for a class
of nonsmooth convex optimization problems. Our starting point was an
equivalent reformulation of this problem in terms of finding the
saddle points of an augmented Lagrangian function. This reformulation
has naturally led us to study the associated saddle-point dynamics,
for which we established convergence to the set of solutions of the
nonsmooth convex program. The novelty of our analysis relies on the
identification of a global Lyapunov function for the saddle-point
dynamics.  Based on these results, we have introduced a discontinuous
saddle-point-like algorithm that enjoys the same convergence
properties and is fully amenable to distributed implementation over a
group of agents that seeks to collectively solve the optimization
problem. With respect to consensus-based approaches, the novelty of
our design is that it allows each individual agent to asymptotically
find its component of the solution by interacting with its neighbors,
without the need to maintain, communicates, or update a global
estimate of the complete solution vector.  We also established the
performance properties of the proposed coordination algorithms for
convex optimization scenarios subject to equality constraints. In
particular, we explicitly characterized the exponential convergence
rate under mild convexity and regularity conditions on the objective
functions. Future work will characterize the rate of convergence for
nonsmooth convex optimization problems subject to both inequality and
equality constraints, study the robustness properties of the proposed
algorithms against disturbances and link failures, identify suitable
(possibly aperiodic) stepsizes that guarantee convergence of the
discretization of our dynamics, design opportunistic state-triggered
implementations to efficiently use the
capabilities of the network agents, and explore the extension of our
analysis and algorithm design to optimization problems defined over
infinite-dimensional state spaces.

\bibliographystyle{siam}
\bibliography{main}

\myclearpage
\section*{Appendix}\label{app}
\renewcommand\thetheorem{A.\arabic{theorem}}

We gather in this appendix various intermediate results used in the
derivation of the main results of the paper. The following two results
characterize properties of the generalized Hessian of $\CCC$
functions. In each case, let $f\in\CCC(\rln,\rl)$ and consider the
set-valued map $\partial(\nabla f):\rln\rightrightarrows\rlnn$ as
defined in Section~\ref{sec2:sub1}. For $x,y\in\rln$, we let
$[x,y]=\{x+\theta(y-x)\mid\theta\in[0,1]\}$ and study the set
\begin{equation*}
	\partial(\nabla f([x,y]))=\bigcup_{z\in[x,y]}\partial(\nabla f)(z).
\end{equation*}

\begin{lemma}[Positive definiteness]\label{app:lm:posdef}
  Let $f\in\CCC(\rln,\rl)$ and suppose $\partial(\nabla
  f)\succ0$. Then, $\co\big\{\partial(\nabla f([x,y]))\big\}\succ0$
  for all $x,y\in\rln$.
\end{lemma}
\begin{proof}
  Since $\partial(\nabla f)(x)\succ0$ for all $x\in\rln$, it follows
  $\bigcup_{z\in[x,y]}\partial(\nabla f)(z)\succ0$. Moreover, since
  the cone of symmetric positive definite matrices is convex itself,
  it contains all convex combinations of elements in
  $\bigcup_{z\in[x,y]}\partial(\nabla f)(z)$, i.e., in particular
  $\co\big\{\partial(\nabla f([x,y]))\big\}\succ0$, concluding the
  proof.
\end{proof}

\begin{lemma}[Compactness]\label{app:lm:compact}
  Let $f\in\CCC(\rln,\rl)$. Then, the set $\partial(\nabla f([x,y]))$
  and its convex closure are both compact.
\end{lemma}
\begin{proof}
  Boundedness of the set $\partial(\nabla f([x,y]))$ follows
  from~\cite[Proposition 5.15]{RTR-RJBW:98}. To show that
  $\partial(\nabla f([x,y]))$ is closed, take
  $\nu\in\cl\partial(\nabla f([x,y]))$. By definition, there exists
  $\{\nu_{n}\}\subset\partial(\nabla f([x,y]))$ such that
  $\nu_{n}\to\nu$. Since $\nu_{n}$ belongs to $\partial(\nabla
  f([x,y]))$, let us denote $\nu_{n}\in\partial(\nabla f)(z_{n})$,
  i.e., $\nu_{n}$ is a vector based at $z_{n}$. Similarly, let $z$ be
  the point at which the vector $\nu$ is based. Following the above
  arguments, we have $z_{n}\to z$. Since $[x,y]$ is compact and
  $z_{n}\in[x,y]$, we deduce $z\in[x,y]$. Assume, by contradiction,
  that $\nu\notin\partial(\nabla f)(z)$. Then, since $\partial (\nabla
  f)(z)$ is closed, there exists $\varepsilon>0$ such that
  $\{\nu\}\cap\partial(\nabla
  f)(z)+\BB(0,\varepsilon)=\emptyset$. Using upper semi-continuity,
  there exists $N\in\mathbb{N}$ such that if $n\geq N$, then
  $\partial(\nabla f)(z_{n})\subset\partial(\nabla
  f)(z)+\BB(0,\varepsilon)$. This fact is in contradiction with
  $\nu_{n}\to\nu$. Therefore, it follows $\nu\in\partial(\nabla
  f)(z)\subset\partial(\nabla f([x,y]))$, and we conclude
  $\partial(\nabla f([x,y]))$ is closed. Thus, $\partial(\nabla
  f([x,y]))$ is compact, and so is $\co\big\{\partial(\nabla
  f([x,y]))\big\}$, concluding the proof.
\end{proof}

\end{document}